\documentclass[10pt]{article}
\usepackage[dvipsnames]{xcolor}
\usepackage{amsfonts,latexsym,amssymb,amsthm,amsmath,graphicx,cases,comment}
\usepackage{paralist}
\usepackage{graphics} 
\usepackage{epsfig} 
\usepackage{epstopdf}

\usepackage[titletoc,title]{appendix}
\usepackage{empheq}
\usepackage{cancel}
\usepackage{tikz}
\usepackage{enumitem}
\usetikzlibrary{arrows,shapes}
\usetikzlibrary{decorations.pathmorphing,decorations.pathreplacing}
\usetikzlibrary{calc,patterns,angles,quotes}

\usepackage{float}

\usepackage{tcolorbox}

\usepackage[colorlinks=true]{hyperref}
\hypersetup{urlcolor=blue, citecolor=red}

\newtheorem{theorem}{Theorem}[section]
\newtheorem{thm}{Theorem}

\newtheorem{lemma}[thm]{Lemma}

\theoremstyle{definition}
\newtheorem{definition}[theorem]{Definition}
\newtheorem{rmk}{Remark}

\newcommand{\be}{\begin{equation}}
\newcommand{\ee}{\end{equation}}
\newcommand{\bsubeq}{\begin{subequations}}
	\newcommand{\esubeq}{\end{subequations}}

\newcommand{\He}{\mathfrak{E}}
\newcommand{\D}[1]{\mathcal{D}(#1)}
\newcommand{\Ls}{L^2(\Omega)}
\newcommand{\p}{A^{1/2}}
\newcommand{\A}{\mathcal{A}}
\newcommand{\E}{\mathcal{E}}
\newcommand{\Ha}{\mathbb{H}_0}
\newcommand{\Hb}{\mathbb{H}_1}
\newcommand{\Hc}{\mathbb{H}_2}

\numberwithin{equation}{section}
\numberwithin{thm}{section}
\numberwithin{rmk}{section}

\allowdisplaybreaks
\begin{document}
	\title{Vanishing relaxation time dynamics of the Jordan Moore-Gibson-Thompson equation arising in  nonlinear acoustics}
	\author{
		Marcelo Bongarti, Sutthirut Charoenphon and Irena Lasiecka\\
		{\small Department of Mathematical Sciences}\\
		{\small University of Memphis}\\
		{\small Memphis, TN 38152 USA}\\
	}
	\date{}
	\maketitle

	\begin{abstract}
		The (third-order in time) JMGT equation  \cite{Jordan2,HCP} is a nonlinear (quasi-linear) Partial Differential Equation (PDE) model introduced to describe a nonlinear propagation  of sound in an acoustic medium. The important feature is that the model avoids the infinite speed of propagation paradox associated with a classical second order in time equation  referred to as Westervelt equation. Replacing Fourier's law by Maxwell-Cattaneo's law  gives rise to the third order in time derivative scaled by a small parameter $\tau >0$, the latter represents the thermal relaxation time parameter and is intrinsic to the medium where the dynamics occur. In this paper we provide an asymptotic analysis of the third order model  when $\tau \rightarrow 0 $. It is shown that the corresponding solutions  converge {\it  in a strong topology of the phase space } to a  limit which is the  solution of Westervelt equation. In addition, rate of convergence is provided for solutions displaying higher order regularity.  This addresses an open question raised in \cite{kaltev2}, where a related JMGT equation has been studied and {\it weak star } convergence  of the solutions when $\tau \rightarrow 0$ has been established. Thus, our main  contribution is showing {\it strong convergence on infinite time horizon,} along with related rates of convergence valid on a finite time horizon.  The key to unlocking the difficulty owns to  a tight control and propagation  of the ``smallness" of the  initial data in carrying the estimates at three different topological levels.  The rate of convergence allows one then to estimate the relaxation time needed for the signal to reach the target. The interest in studying this type of problems is motivated by a large array of applications arising in engineering and medical sciences. 

	\end{abstract}
\section{Introduction}
Propagation of  nonlinear waves in an acoustic environment  has been a topic of great interest and activities. Broad range of physical applications
including  ultrasound technology, welding, lithotrips, thermotherapy, ultrasound cleaning, and sonochemistry \cite{kalt,crig, ap2,ap3,ham} are just a few examples. In view of this, it is not surprising that mathematical  models are of great interest and became a highly active field of research.
 We will be considering  the   Jordan-Moore-Gibson-Thompson (JMGT) equation, which, although simple, does display several mathematical intricacies of interest in PDE area and  it is representative of the underlying physics.  The JMGT equation written in the variable $\psi$, which denotes velocity potential, can be recasted  as \begin{equation}\label{jmgt00}
 \tau \psi_{ttt} + \psi_{tt} - c^2\Delta \psi - b^\tau \Delta \psi_t = \frac{d}{dt} [ k(\psi_t)^2]
 \end{equation}
 where $c$ denotes the speed of sound, $\delta >0$ diffusivity  of the sound ($b^\tau = \delta + \tau c^2$), $\tau > 0 $ thermal relaxation time parameter and $k$ a parameter of nonlinearity, see \cite{Jordan2,Jordan0,kalt} and references therein for derivation of the model. Rewriting equation in terms of the pressure $u \sim \rho_0 \psi_t $ ($\rho_0$= mass density) leads to 
 \begin{equation}\label{jmgt1} 
\tau u_{ttt} + (1-2ku)u_{tt} - c^2\Delta u - b \Delta u_t = 2k(u_t)^2.
\end{equation}
 The presence of the  constant $\tau$, which accounts for finite speed of propagation, removes the so called infinite speed of propagation paradox, see \cite{ap2,ap3,ap4,ap5}. In fact, when $\tau =0$ the corresponding PDE becomes the Westervelt equation -- which is of parabolic type -- and with $\tau >0 $ the system is hyperbolic-like and its linear version represents a group. 
  Since the parameter $\tau>0$ is relatively small, it is essential  to understand the effects of diminishing values of relaxation. This is a particularly delicate issue since the $\tau$-dynamics is governed by a generator which is {\it  singular}  as $\tau\rightarrow 0.$ 
The goal of this paper is to consider the vanishing parameter $\tau \rightarrow 0 $ and its consequences on the  resulting evolution. 
Accordingly  we will show that the Westervelt equation,
\begin{equation}\label{West}
(1-2ku)u_{tt} -c^2\Delta u - \delta\Delta u_t = 2k(u_t)^2
\end{equation}
is a {\it strong}  limit of the JMGT equation, when the relaxation parameter vanishes. In this spirit we shall answer  an open question raised in recent manuscript \cite{kaltev2}, where {\it weak star} convergence has been shown for a related model. In addition, the quantitative  rate of convergence of the corresponding solutions will also  be derived.  
As we shall see, the key in  unlocking the difficulty is a good control of  topological {\it smallness} of the initial data  along with  an appropriate calibration   of the  estimates at various topological levels. 
\ifdefined\xxxxx
\subsection{Physical motivation, modeling and thermal relaxation parameter}
Based on recent   developments in modeling of nonlinear acoustic waves \cite{HCP, ham, kalt,kuz, stokes}  the physical model can be described with the main physical quantities being
$\vec{v}=$ the acoustic particle velocity, $p=$ the acoustics pressure and $\varrho=$ the mass density which can be decomposed as
$$\vec{v}=\vec{v}_0+\vec{v}_\sim,\quad p=p_0+p_\sim\quad \varrho=\varrho_0+\varrho_\sim.$$

The equations governing the propagation of sound in a fluid medium are;
\begin{itemize}
	\item the Navier Stokes equation where $\varsigma\mathtt{v}$ is the bulk viscosity and $\mu\mathtt{v}$ is the shear viscosity
	\begin{equation}\label{nv}
	\varrho\big(\vec{v}_t+(\vec{v}\cdot\bigtriangledown)\vec{v}\big)+\bigtriangledown p=\mu_\mathtt{V}\Delta\vec{v}+\big(\dfrac{\mu_\mathtt{V}}{3}+\varsigma_\mathtt{V}\big)\bigtriangledown(\bigtriangledown\cdot\vec{v});
	\end{equation}
	\item the equation of continuity
	\begin{equation}\label{ec}
	\bigtriangledown\big(\varrho\vec{v}\big)=-\varrho_t;
	\end{equation}
	\item the state equation of relation between $p_\sim$ and $\varrho_\sim$
	\begin{equation}\label{rb}
	\varrho_\sim=\dfrac{p_\sim}{c^2}-\dfrac{1}{\varrho_0c^4}\dfrac{B}{2A}p^2_\sim-\dfrac{k}{\varrho_0c^4}\big(\dfrac{1}{c_\mathtt{V}}-\dfrac{1}{c_p}\big)p_{\sim t},
	\end{equation}
\end{itemize}
where $\dfrac{B}{A}$ is the nonlinear parameter. We subtract the divergence of \eqref{nv} from the time derivative of \eqref{ec}, add \eqref{rb} and neglect the third and higher order terms. We arrive at the Kuznetsov equation. Then the Westervelt equation is obtained by omitting the quadratic velocity term.
\begin{equation}\label{wes}
\dfrac{1}{c^2}u_{tt}-\Delta u-\dfrac{b}{c^2}\Delta(u_t)=\dfrac{\beta_a}{\varrho_0c^4}(u^2)_{tt},
\end{equation}
where $\beta_a=1+\dfrac{B}{2A},$ $u$ denotes the acoustic pressure fluctuations, $c$ the speed of sound, $d$ the diffusivity of sound, $\varrho_0$ the mass density, and $\varrho_0v_t=-\bigtriangledown u.$ Next, 
let $u$ denotes the pressure. Then the potential degeneracy of \eqref{wes} is expressed as
\begin{equation}\label{kz}
(1-2ku)u_{tt} -c^2\Delta u - \delta\Delta u_t = 2k(u_t)^2,
\end{equation}
where $k=-\dfrac{\beta_a}{c^2}$ and $b=\delta+\tau c^2.$ Then P.M. Jordan \cite{jordan2} extended these models based on the original derivation
\begin{equation}
\bigg(\dfrac{d}{dt}+q\bigg)\dfrac{d^2s}{dt^2}=k\biggl\{(1+\alpha\beta)\dfrac{d}{dt}+q\biggr\}\dfrac{d^2s}{dx^2},
\end{equation} 
obtained by G.G. Stokes \cite{stokes} in 1851 to a higher order PDE which is a \textit{third-order} in time PDE model, where $\psi$ denotes the acoustic velocity potential and $\tau$ is the positive constant accounting for relaxation time. We obtain
\begin{equation}
\tau\psi_{ttt}+\psi_{tt}-c^2\Delta\psi-b\Delta\psi=-\bigg(\dfrac{\beta_a}{c^2}(\psi_t)^2\bigg)_t.
\end{equation}

The fact that the model is of third-order in time results from application of the Maxwell-Cattaneo law \cite{ap2, ap3, ap4, ap5}, rather than the more traditional Fourier law in describing heat conductivity. This is to avoid the so-called infinite speed of propagation  paradox -associated with the  Fourier's law. Maxwell -Cattaneo law introduces a ``small" parameter 
$\tau.$

Synthesizing the model from the mathematical-semigroup standpoint we are dealing with the following nonlinear system: this is {\it third-order } in time equation which is a nonlinear (quasi-linear) Partial Differential Equation (PDE) model used to describe the acoustic velocity potential in ultrasound wave propagation. It is referred as  Jordan Moore Gibson Thompson (JMGT) equation
\begin{equation}\label{jmgt}
\tau u_{ttt} + (1-2ku)u_{tt} - c^2\Delta u - b \Delta u_t = 2k(u_t)^2,
\end{equation}
where $u=u(t,x)$ is the acoustic pressure, $k$ is a parameter that depends on the mass density and the constant of the nonlinearity, the parameters $c$ and $b$ denote the speed and diffusivity of the sound, respectively. They are required to be positive. There are boundary conditions associated with the model-say homogeneous Dirichlet boundary conditions. 
There are two immediate features of interest in the model as follows.
\begin{enumerate}[label=(\alph*)]
	\item The presence of the parameter $\tau > 0 $ which corresponds to time relaxation and  makes the problem third order in time.
	\item A possibly degenerate character due to the uncontrolled sign of $ (1 - 2 ku ) $ where $u$ is the unknown solution. 
\end{enumerate}
\fi 

 A mathematical interest in third order equations  stems also from the fact that an existence of semigroup for the \textit{linearization}  fails when the  diffusivity $b =0$  \cite{fatto}.  On the other hand,  on physical grounds the parameter $\tau $  accounts for physically relevant   finite speed of propagation of the waves. Thus, the analysis needs to reconcile ``small" amplitude waves with the limit process.  The main question to contend with is ``how small is small ". This leads to  a string of estimates  with a minimal requirement imposed on the  "smallness".  The latter  is  the key in unlocking the difficulty with  {\it strong} convergence. 
\subsection{The model and related  literature}

Let $\Omega$ be a bounded domain in $\mathbb{R}^n$ ($n=2,3$) with a $C^2-$boundary $\Gamma = \partial \Omega$ immersed in a resting medium.
 Let  $A: \D{A} \subset \Ls \to \Ls$ be the unbounded, positive self-adjoint and densely  defined operator defined by the action of  the Dirichlet Laplacian, i.e., $A = -\Delta$ with $\D{A} = H_0^1(\Omega) \cap H^2(\Omega)$.

Let $\delta , c >0$ denote diffusivity and speed of the sound, respectively; $k > 0$ denotes a nonlinear parameter and $b^\tau := \delta + \tau c^2$. Let $T>0$ (could also be $T = \infty$) and the relaxation parameter $\tau \in \Lambda = (0, \tau_0]$ for some $0 < \tau_0 \in \mathbb{R}$ that will be taken as small as needed, but fixed. We then consider a family 
of third-order (in time)  Jordan-Moore-Gibson-Thompson (JMGT) equations \begin{equation}\label{eqnl} \tau u_{ttt}^\tau + (1-2ku^\tau)u_{tt}^\tau + c^2A u^\tau + b^\tau A u_t^\tau = 2k(u_t^\tau)^2 \ \ \mbox{on}  \ \ (0,T) \end{equation} with the  initial conditions:
$u^\tau(0)=u_0,u_t^\tau(0)=u_1, u_{tt}^\tau(0) = u_2.$
The wellposedness theory, [both local and global] for the model \eqref{eqnl}  has been well developed  and known by now for each value of $\tau > 0$.  This also includes regularity theory where for sufficiently smooth and compatible initial data one obtains smooth solutions \cite{jmgt,kalt}.  For reader's convenience  we shall recall some of the relevant results below. However, the focus of this paper is on the asymptotic analysis 
when $\tau \rightarrow 0 $.  


{\bf Notation:} {Throughout this paper, $L^2(\Omega)$ denotes the space of Lebesgue measurable functions whose squares are integrable and $H^s(\Omega)$ denotes the $L^2(\Omega)$-based Sobolev space of order $s$. We denote the inner product in $L^2(\Omega)$ and by $(u,v) = \int_\Omega uvd\Omega$ and the respective induced norm is denoted by $\|\cdot\|_2.$ If less mentioned normed spaces appear, their norms will be indicated with a sub index, i.e., the norm of a space $Y$ will be denoted by $\|\cdot\|_Y.$  Finally, we denote by $\mathcal{L}(Y)$ the space of bounded linear operators from $Y$ to itself equipped with the uniform norm.
$B(X)_r$ denotes a ball of radius $r$ in a Banach space $X$. The projection $P:\mathbb{R}^3\rightarrow \mathbb{R}^2 $ is defined as $P(a,b,c) = (a,b) $. 
Various constants (generic) will be denoted by letters  $C,c, C_i$ -- they may be different in different occurrences. $C(s)$ denotes a continuous function for $s \geq 0$ and such that $C(0) =0 $.}

We begin  with collecting relevant  results   related to  the  wellposedness of solutions to \eqref{eqnl} for each value of the parameter $\tau > 0$.  
Recalling that $\D{\p} = H_0^1(\Omega)$   we define $\mathbb{H}_0, \mathbb{H}_1, \mathbb{H}_2$ as  follows $$\mathbb{H}_0 \equiv \D{\p} \times \D{\p} \times L^2(\Omega);~\mathbb{H}_1 \equiv \D{A} \times \D{\p} \times L^2(\Omega);~\mathbb{H}_2 \equiv \D{A} \times \D{A} \times \D{\p},$$ with $\D{A}$ and $\D{\p}$ being equipped with the graph norm and the product spaces $\mathbb{H}_i$ ($i = 0,1,2$) equipped with induced euclidean norms. 

We rewrite the abstract equation \eqref{eqnl} -- in the variable $U^\tau = (u^\tau, u_t^\tau, u_{tt}^\tau)^\top$ -- as the first-order system
\begin{equation}\begin{cases}U^\tau_t(t) = \A^\tau U^\tau(t)+ \tau^{-1} F(U^\tau), \ t>0, \\ U^\tau(0)=U_0=(u_0,u_1,u_2)^\top \end{cases}\label{AbS1}\end{equation}
where $\A^\tau:= M_{\tau}^{-1}\A^0$ with $\A^0$  and $M_{\tau} $  given by \begin{equation} \A^0:= \left(\begin{array}{ccc}
0  & 1 & 0 \\
0 & 0 & 1 \\
- c^2A & -b^\tau A & -1
\end{array}\right) ; ~~ \ M_{\tau}:= \left(\begin{array}{ccc}
1  & 0 & 0 \\
0 & 1 & 0 \\
0 & 0 & \tau
\end{array}\right) ;~~F(U^\tau) :=  \left(\begin{array}{c} 0\\0\\  2k[(u_t^\tau)^2+u^\tau u_{tt}^\tau] \end{array} \right). \label{mop}\end{equation}  

\

\begin{rmk}
 Notice that the operator $\A^\tau$ becomes singular as $\tau \to 0$. 
 \end{rmk}
 We shall also introduce spaces $\mathbb{H}^{\tau}_i \equiv M_{\tau}^{1/2} \mathbb{H}_i $ for $ i=0,1,2$ with the  topology 
 $\|U\|_{\mathbb{H}_i^{\tau}} = \|M_{\tau}^{1/2} U\|_{\mathbb{H}_i}$.

With the above setting, we are able to introduce the notion of solution for the nonlinear problem. Denote by $\{T^\tau(t)\}_{t \geqslant 0}$ the $C_0$-semigroup associated with the operator $\A^{\tau} $ for each value of $\tau > 0 $.

This semigroup  exists on each space  $\mathbb{H}_i^\tau , i=0,1,2$. \begin{definition}[\cite{jmgt}]
	We say that $U^\tau(t) = (u^\tau(t),u_t^\tau(t),u_{tt}^\tau(t))^\top$ is a mild solution of \eqref{eqnl} on $[0,T)$ provided 
		 $U^\tau \in C(0,T; \mathbb{H}_i)$ and 
		 $U^\tau(t)$ satisfies the following integral equation 
		$$ U^{\tau} (t) = T^\tau(t) U(0) + \tau^{-1} \int_0^t T(t-s)F(U^\tau)d\sigma $$ 
\end{definition}

\ifdefined\xxxxx

As pointed before, quasi-linear theory requires restrictions on the initial data for obtaining wellposedness and stability of solutions. This, in general, includes requiring smoothness of the initial data (considering that a regularity result of the type `regular IC' implies `regular solution' is established) and also smallness. Although we strongly need regularity in order to carry out estimates and justify them, we aim to improve the approach here by assuming smalless only in the lowest possible topology, which is the one where the linearized problem is wellposed and uniformly stable under some conditions on the parameters. See Remark \ref{topr} for a contextualization under the optics of other works from the literature.
\fi
The linear problem  can be recasted  by taking and $F(U) =0 $ and its wellposedness along with  uniform exponential stability -- on both $\mathbb{H}_0$ and $\mathbb{H}_1$ -- was obtained in \cite{TrigMGT} under the condition (for stability) that $\gamma^\tau := 1 -  \tau c^2b^{-1} >0$. Given a solution $U^\tau \in \Ha^\tau$  its energy is defined via the $u^\tau$-dependent functional \begin{equation}\label{h0en}
E^\tau(t) := E_0^\tau(t) + E_1^\tau(t),~where~
E_0^\tau(t) = \dfrac{\alpha}{2}\left\|u^\tau_t(t)\right\|_2^2+ \dfrac{c^2}{2}\left\|\p u^\tau(t)\right\|_2^2
\end{equation} and, with $z^\tau(t) := u^\tau_t(t) + \dfrac{c^2}{b}u^\tau(t) \in \D{\p}$,
$E_1^\tau(t) \equiv \dfrac{b}{2}\left\|\p z^\tau(t)\right\|_2^2 + \dfrac{\tau}{2}\left\|z^\tau_t(t)\right\|_2^2 + \dfrac{c^2\gamma^\tau}{2b}\left\|u^\tau_t\right\|_2^2.
$

For  $U^\tau$ in $\Hb^\tau$, the  corresponding energy is given by
\begin{equation}
\mathcal{E}^\tau(t) := E^\tau(t) + \|Au^\tau(t)\|_2^2 \label{h1en},
\end{equation} We also introduce the energy associated with solutions in $\Hc$: $\He^\tau(t) := \E^\tau(t) + \|Au^\tau_t(t)\|_2^2 + \tau \|\p u^\tau_{tt}(t)\|_2^2.$ 

The  following local and global wellposedness results are known \cite{jmgt} : \begin{theorem}[\bf local (in time) wellposedness in $\mathbb{H}_1$, \cite{jmgt}]\label{esh11}
	Let $T>0$ be arbitrary and $\gamma=1-\dfrac{c^2\tau}{b}$. There exists $\rho_{T,\tau} (\gamma)>0$ such that if the initial data $U_0$ satisfies $\E^\tau(0)\leqslant\rho_{T,\tau} (\gamma),$ then there exists a unique solution  such that
	$$\E^\tau(t)< C_{\tau,T}(\E^\tau(0)) < \infty,, ~for~all~ t\in[0,T) $$
	The said solution depends continuously on the initial data  in $\mathbb{H}_1$ topology. An alternative statement holds  for all $U_0 \in \mathbb{H}_1 $  and 
	small time $ T >0$.
\end{theorem}
\begin{theorem}[\bf global (in time) wellposedness, \cite{jmgt}]\label{esh12}
	Let $\gamma>0$. Then,  there exists $\rho_{\tau}(\gamma)>0$ such that if the initial data $U_0$ satisfies $\E^\tau(0)\leqslant\rho_{\tau}(\gamma),$ then there exists a unique solution  such that
	$$\E^\tau(t)\leqslant C_{\tau}(\E^\tau(0) ), t >0 ~and~ \E^\tau(t) \leqslant C_{\tau} (\rho_{\tau} ) e^{-\omega t} ,  t>0, \omega >0 $$
\end{theorem}
\ifdefined\xxxxx
 \begin{rmk}
	\label{rmkwph2}The same sequence of results given in Theorems \ref{esh11}, \ref{esh12} also holds -- with some adaptations on the estimates -- for the space $\Hc$. In fact, the chain of estimates developed in the later Sections \ref{unlow} and \ref{unhigh} serve as \emph{a priori} estimates used to build up a Fixed Point Theory that guarantees the existence of solutions. In fact, these estimates allow the solutions to be construct under smallness of the initial data being taken with respect to the $\Ha$-topology.
\end{rmk}
\fi
The results above  show, in particular, that there exists a nonlinear flow  defined  globally in time  on a small ball (small radius) in $\mathbb{H}_1.$ In addition, the said solutions display additional regularity, provided the initial data are more regular as we make precise below. 
\begin{theorem}[\bf local (in time) wellposedness in $\mathbb{H}_2$]\label{wph2}
	Let $T>0$ be arbitrary and $\gamma=1-\dfrac{c^2\tau}{b}$. There exists $\rho = \rho_{T,\tau} (\gamma)>0$, such that if the initial data $U_0\in \mathbb{H}_2$ and it  satisfies $E^\tau(0)\leqslant\rho_{T,\tau} (\gamma),$ then there exists a unique solution  $U(t)$  such that
	$$\He^\tau(t)< C_{\tau,T}(\He^\tau(0)) < \infty,, ~for~all~ t\in[0,T) $$
	The said solution depends continuously on the initial data  in $\mathbb{H}_2^\tau$ topology.  In addition,  for each $\tau > 0 $, $u_{ttt}^{\tau} \in C([0,T;] L_2(\Omega))$.  An alternative statement holds  for all $U_0 \in \mathbb{H}_2 $  and 
	small time $ T >0$.
\end{theorem} \begin{proof} The result of Theorem \ref{wph2}  is proved by using a fixed-point argument with the techniques reminiscent to the proof of Theorem \ref{esh11}. For completeness, we provide the proof  in the appendix. 
\end{proof}

\begin{rmk}[\bf wellposedness with data small in lower topology] Theorem \ref{esh11}
 was stated exactly how it was stated in \cite{jmgt}, however, as in Theorem \eqref{wph2}, we can slightly modify the proof in \cite{jmgt} to obtain wellposedness and $\mathbb{H}_1$-exponential stability asking for smallness of the data only in $\mathbb{H}_0.$ The strategy is the same as in the proof of Theorem \ref{wph2}, see Appendix \ref{app}. \end{rmk}

The problem of interest  in this paper is an asymptotic analysis  of solutions when the relaxation parameter $\tau$ tends to zero. It is expected that the limit solution  
should coincide  (formally)  with  the solution of the Westervelt equation (Equation \eqref{eqnl} with $\tau = 0$). 
In  \cite{bongarti}  asymptotic analysis  has been performed for  the   linearized equation \eqref{eqnl}. 
It was shown there that  the linear semigroups $U^{\tau} (t)$   converge strongly, in the  topology of both $\mathbb{H}_i, (i =0,1)$ to the limit $U^0(t)$ -- solution of  strongly damped wave equation. Rates of convergence have been also  shown for finite time horizon under additional regularity assumptions. The difficulty  encountered  lies in the fact that the operator $\mathcal{A}^\tau$ is singular when $\tau\rightarrow 0.$ As a consequence,  more standard techniques based on Trotter Kato type theorems \cite{kato}  cannot be applied.  However, the difficulty  was handled by playing  strategically two levels of the estimates 
with one allowing to derive the  convergence rates. 

The aim of the present  work is to obtain related result valid for the nonlinear model \eqref{eqnl}. This entails to proving 
 {\bf (1)} Uniform (in $\tau$) exponential decay of nonlinear flows; {\bf (2)}   Strong Convergence of the flows  (in an appropriate sense)  with respect to vanishing relaxation parameter $\tau \geq 0$.
 As to the first task, this is routed -- as expected -- in a string of various estimates obtained via  the use of  multipliers. Careful analysis of ``smallness" requirement for the initial data  is critical in this step and also leads to new (more refined)  results  in the analysis of JMGT  equations for a fixed $\tau > 0$.    As to the second task, the main difficulty is topological compatibility of  ``smallness"  of initial data (typical in quasilinear problems)  required by the existence theory  and the passage to the limit via density, where such ``smallness" may not be retained. To unlock the difficulty, we were able to construct solutions with much less stringent assumptions on the said ``smallness" which is just sufficient (it  requires only  the $\mathbb{H}_0$ topology) to carry the limit argument. 
 This, in turn, entails to a new ``look" at the energy estimates (Step 1) where such ``reduced smallness" is  traced and propagated. 
\ifdefined\xxxxxx
 The strategy to overcome this difficulty lies in a very simple observation along with standard ideas -- this regards to the mass operator $M_{\tau}$ introduced in equation \eqref{mop} -- already widely used in mechanics. First observe that the energy functional $E^\tau(t)$ is such that $$E^\tau(t) \approx \|\p u^\tau(t)\|_2^2 + \|\p u_t^\tau(t)\|_2^2 + \tau\|u^\tau_{tt}\|_2^2$$ and such equivalence is $\tau$-independent. Then, in the spirit of this remark, we define the partly $\tau$-weighted space $\Ha^\tau$ -- similar renormalizations will be carried for $\Hb^\tau$ and $\Hc^\tau$ -- as \begin{equation}
\Ha^\tau := \left\{M_\tau^{1/2} V; \ V \in \Ha^\tau\right\}
\end{equation} equipped with the norm $\Ha \ni V \mapsto \|V\|_{\tau,0} := \|M_{\tau}^{1/2} V\|_{\Ha}.$ The norms on $\Hb^\tau$ and $\Hc^\tau$ are defined the same way and denoted by $\|\cdot\|_{\tau,1}$ and $\|\cdot\|_{\tau,2}$, respectively.

The literature for the asymptotic (in $\tau$) analysis of (J)MGT equation, to the best of the authors knowledge, consists of three works. First, a fundamental first step for the research developed in this paper, the same authors proved in \cite{bongarti} that under the setting of the partly $\tau$-weighted spaces, the results of linear wellposedness and uniform linear stability can be made uniform in $\tau.$ Precisely, beyond the generation of the semigroup solutions on $\mathbb{H}_i \ (i = 0,1,2)$, it was shown uniform (in $\tau$) stability of linear solutions as stated below (see \cite{bongarti}).

\begin{tcolorbox}[title=Linear $\tau$-asymptotic analysis,colframe=green!75!black]
	\begin{theorem}[uniform (in $\tau$) exponential stability]
		Consider the family $\mathcal{F} = \{T^\tau(t)\}_{\tau>0}$ of groups generated by $\A_\alpha^\tau$ on $\Ha$.  Assume that $\gamma^\tau \equiv \alpha - c^2\tau (b)^{-1}\geq \gamma_0 >0 $. Then, there exists $\tau_0 > 0$ and constants $M = M(\tau_0),\omega = \omega(\tau_0)>0$ $($both independent of $\tau)$ such that $$\left\|T^\tau(t)\right\|_{\mathcal{L}(\mathbb{H}^{\tau}_0)} \leqslant M e^{-\omega t} \ \mbox{for all} \ \tau \in (0,\tau_0] \ \mbox{and} \ t \geqslant 0.$$
	\end{theorem} 

The exact same result also holds for $\Hb$ with respect to the energy $\E^\tau(t)$ and for $\Hc$ with respect to the energy $\mathfrak{E}^\tau(t)$ defined as $$\mathfrak{E}^\tau(t) := \E^\tau(t) + \left\|Au^\tau_t(t)\right\|_2^2 + \tau \left\|\p u^\tau_{tt}(t)\right\|_2^2 \approx \left\|(u^\tau(t),u^\tau_t(t),u^\tau_{tt}(t))\right\|_{\tau,2}^2.$$
\end{tcolorbox}

The above allows to construct a limit semigroup $\{T(t)\}_{t \geqslant 0}$ corresponding to $\tau =0$ on $\mathbb{H}_{0}^0 := \D{\p} \times \D{\p}$ (and also in $\mathbb{H}_1^0 := \D{A} \times \D{\p}$) which is analytic because of the strong damping. This is in strong contrast to the semigroups $\{T^{\tau}(t)\}_{t \geqslant 0}, \ \tau > 0$ -- which are of hyperbolic type and time reversible (group). The following convergence results were the main results in \cite{bongarti}.

\begin{tcolorbox}[title = Convergence of linear semigroups,colframe=green!75!black]
	
	Let $P$ denote the projection on the two first coordinate, i.e., $P(a,b,c) = (a,b).$ 
\begin{theorem}[rate of convergence]\label{tlc1}
	Let $U_0 \in \mathbb{H}_2.$ Then there exists $C = C(T, \tau_0)$ such that  $$\left\|P T^\tau(t) U_0 - T(t)PU_0\right\|_{\mathbb{H}_0^0}^2 \leqslant C \tau |U_0|_{\tau,2}^2,$$  uniformly for  $ t \in [0,T].$ The same works replacing $\Ha^0$ by $\Hb^0.$
\end{theorem}
\begin{theorem}[strong convergence]\label{tlc2}
	Let $U_0 \in \mathbb{H}_0$. Then  the following strong convergence holds. \begin{equation}\label{convn1}
	\left\|P T^\tau(t)U_0  - T(t)PU_0\right\|_{\mathbb{H}_0^0 } \to 0 ~as~\tau \to 0,
	\end{equation} uniformly for all $t \geqslant 0.$ The same works replacing $\Ha^0$ by $\Hb^0.$
\end{theorem}
\end{tcolorbox}
\fi

This brings us to the very recent works  \cite{kaltev2} and \cite{kaltev1} which study the JMGT--Kuznetsov model for acoustic velocity potential (whereas here we look for the JMGT model for acoustic pressure):
 $\tau\psi_{ttt} + (1-2k\psi_t)\psi_{tt}  - c^2 \Delta \psi - b\Delta \psi_t = 2\nabla\psi\nabla\psi_t$.
 For this model, the authors derive {\it weak star} convergence of solutions $\psi$  to the corresponding limit when $\tau =0$. Because of an  extra derivative in the equation, the topology of convergence is $\mathbb{H}_2$ with the associated ``smallness" requirement  imposed on the initial data. Remark 7.2 in \cite{kaltev2} raises an open question whether this convergence can be improved to the {\it strong} one. Numerical results  presented in  \cite{kaltev2} support the conjecture.  The result of the present  paper appears to confirm validity of this conjecture. Although the equation we study is not exactly the same -- the topological considerations are related. 
 \ifdefined\xxxxx
  -- whose basic energy estimates are not covered by the ones derived on \cite{jmgt} -- and proved a chain of results including wellposedess of both linear and nonlinear equations as well as their corresponding ``$\tau = 0$" version so studying the convergence $\tau \to 0$ would make sense. Notice that, adding to the fact that this equation models acoustic velocity potential and has the presence of the nonlinear quadratic gradient -- which requires higher order energy estimates --, many other features differ from our results presented in this paper. From the technical point of view, the Garlekin method is used to obtain wellposedness of the linearized problem and the existence of solutions for the nonlinear counterpart -- relying on smallness and certain regularity of the data - are proved locally, but not global. For the nonlinear part of the limiting equation (the third-order in time), a combination of Shauder's Fixed Point theorem and the classical Banach's Contraction Principle was used, the reason was that, with the given regularity, the constructed self-mapping would (for the Weltervelt's case, i.e., right hand side equals to zero) be contractive only on a $\tau$-dependent neighborhood, which would be an obstacle in the study of the limit as $\tau \to 0.$ Moreover, the Westervelt equation ($\tau = 0$) was obtained as a weak limit of the JMGT equation, while our approach and setting allow us to conclude strong convergence of the underlying nonlinear semigroups. It's also interesting to notice that the authors used norms which are partly $\tau$-dependent to conclude uniform (in $\tau$) estimates.

The same authors (as in \cite{kaltev2}) considered in \cite{kaltev1} the equation $$\tau\psi_{ttt} + (1-2k\psi_t)\psi_{tt}  - c^2 \Delta \psi - b\Delta \psi_t = 2\nabla\psi\nabla\psi_t$$ under mixing Neumann Boundary conditions (nonhomogeneous for a part of the boundary and absorbing for the other) along with homogenous initial data. The same problem of wellposedness and vanishing (in $\tau$) limit was considered. The authors obtained local (extendable to global) nonlinear wellposedness and proved weak convergence of solutions in $\tau.$ \begin{rmk}\label{topr}
\fi 
	A key feature  to emphasize is that references \cite{kaltev1} and \cite{kaltev2} provide {\it  weak star convergence} of solutions for the initial data  which are small in $\He^\tau$ topology. In contrast, our approach provides not only {\it strong} convergence, but  also the  initial data are required  to be small in a lowest topology $E^\tau$ and only bounded in $\E^\tau$. 
	In addition,  {\it rates of convergence}  of classical solutions on every finite time interval are also established. On the other hand, the model studied in the present paper accounts for Westervelt nonlinearity rather than Kuznetsov. It is anticipated that the methodology  developed in this paper would allow  for an  extension to  Kuznetsov's semilinear  terms.
	
	In closing the introduction, we  list several other mathematical works  dealing with  various variants of JMGT model.  Considerations related to regularity -- including boundary regularity of  linear models -- have been studied in \cite{bucci2017,bucci2020}. MGT models under the effects of memory  have been studied \cite{pata,mem3}. 
	The detailed analysis of regularity and decay  on $\mathbb{R}^3 $, by using Fourier analysis, has been conducted in \cite{pellicer2017}. 
\subsection{Main Results }

\ifdefined\xxxxxxxxx
Our research in this paper is formulated as follows:
\begin{enumerate}
	\item
	Uniform (in $\tau$) stability theory -- within a proper functional analytic PDE setting, for solutions corresponding to \eqref{eqnl} for each $\tau$ with quantitative analysis of dependence on $\tau.$ This will be accomplished  by  energy estimates -- multipliers based -- in order to construct an adequate quasilinear theory and trace dependence on $\tau.$ 
	
	\item
	Develop an asymptotic analysis of said solutions when the limiting parameter $\tau$ tends to zero. In particular, determine the limiting behavior and show how fast solutions are converging to this limit state. This particular aspect is particularly important since the values of time relaxation parameter are small and it is critically important to detect potential degeneracies of solutions of the model in this vanishing-limiting region. For this we shall use the estimates generated in the process of achieving goal $\#1$ in order to develop an appropriate nonlinear semigroup theory adequate for the characterization of our solutions and their limits. 
\end{enumerate}
\fi
As anticipated in the introduction, our first goal is to obtain ``good" stability estimates for the $\tau$ problem with the constants independent of $\tau$. 
However, the main challenge is to obtain these estimates with a minimal requirement  of   topological smallness imposed on the initial data. The latter challenge is  stimulated by applicability of density argument where  the size restrictions on the data becomes a predicament.
The corresponding results, at two different topological levels, are formulated  below. 
	\begin{theorem}\label{thm1}
		Let $U_0\in\mathbb{H}_1$ and  assume $\|U_0\|_{\mathbb{H}_1^{\tau}}\leqslant r_1$ for  $r_1>0.$  Then, there exists $\rho_1(r_1) > 0$ \emph{(}independent of $\tau$\emph{)} such that 
		 for 
		$\|U_0\|_{\mathbb{H}_0^{\tau} } \leq \rho_1$,  
		there exist $N_1=N_1(r_1,\rho_1) $and $\omega_1>0$ $($independent of $\tau \in \Lambda = (0, \tau_0]$$)$ such that
		$$\E^\tau(t)\leqslant N_1(r_1,\rho_1)e^{-\omega_1 t},$$
		for all $t> 0$ with  $U^\tau=(u^\tau, u_t^\tau, u_{tt}^\tau)$  a mild solution of \eqref{eqnl} on $\Hb.$
	\end{theorem} 


\begin{theorem}	\label{thm2}
	Let $U_0\in\mathbb{H}_2$  and assume $\|U_0\|_{\mathbb{H}_2^{\tau}}\leqslant r_2$ for  $r_2>0.$ Then,there exists $\rho_2(r_2) $ $($independent of $\tau $$)$  such that for  any   $\|U_0\|_{\mathbb{H}_0^{\tau}}  \leq \rho_2$,  
 there exists $N_2=N_2(r_2,\rho_2) $ and $\omega_2>0$ $($independent of $\tau \in \Lambda = (0, \tau_0]$$)$ such that
	$$\mathfrak{E}^\tau(t)\leqslant N_2(r_2,\rho_2)e^{-\omega_2 t},$$
	for all $t> 0$ with  $U^\tau=(u^\tau, u_t^\tau, u_{tt}^\tau)$ the  solution of \eqref{eqnl} on $\Hc.$
\end{theorem} 

Theorems \ref{thm1} and \ref{thm2} allow the construction of a nonlinear flow  corresponding to the limit $\tau = 0.$ We denote this flow,  as $T^\tau(t) = U^\tau(t,U_0)$ and read: the nonlinear flow corresponding to the initial data $U_0.$

Our  final  result  is to show that the Westervelt equation (see \cite{WE}) \begin{equation}\label{jmgt0}
(1-2ku^0) u^0_{tt} + c^2 A u^0 + \delta A u^0_t = 2k(u^0_t)^2,
\end{equation} with the  initial conditions $u^0(0,\cdot) = u_0, \ u^0_t(0,\cdot) = u_1$ is a limit of the JMGT (equation \eqref{eqnl}) when the thermal relaxation parameter vanishes ($\tau \to 0$). Denoting $ U^0(t,V_0)$ the nonlinear flow  for equation \eqref{jmgt0} correspoding to initial data $V_0$ and recalling  $P$ the projection on the first two coordinates (i.e. $P(a,b,c) = (a,b)$, we have :\begin{theorem}\label{thm3}
	\begin{enumerate}
		\item[\emph{a)}] \emph{\textbf{[Rate of Convergence:]}} Let $T > 0$ and  let $U_0\in\mathbb{H}_2$ with $ \|U_0\| _{\mathbb{H}_0^{\tau}} \leq \rho $ for $ \rho  = min\{ \rho_1,\rho_2\}$ $($$\rho_1,\rho_2$ from Theorems \ref{thm1} and \ref{thm2} respectively$)$. Then there exists a $\tau$-independent constant  $C_T$ such that 
		$$\left\|P (U^\tau(t,U_0))- U^{0}( t,PU_0)\right\|_{\D{A} \times \D{\p}}^2 \leqslant  C_T  \tau \|U_0\|_{\mathbb{H}_2^{\tau}}^2$$  uniformly $($in $t$$)$ for  $ t \in [0,T].$
		\item[\emph{b)}] \emph{\textbf{[Strong Convergence:]}} Let $U_0 \in \mathbb{H}_1$  with $\|U_0\| _{\mathbb{H}_0^{\tau} } \leq \rho $ for $\rho$  as above. Then the following strong convergence 
		$$\left\|P (U^\tau(t,U_0))- U^{0}(t,PU_0)\right\|_{\D{A} \times \D{\p}} \rightarrow0 \ \ \mbox{as} \ \ \tau\rightarrow0$$ 
		holds uniformly on $[0,\infty).$
\end{enumerate}\end{theorem}
The proofs of the theorems \ref{thm1}, \ref{thm2}, \ref{thm3} are given in the subsequent sections. In what follows, we shall  make few comments regarding the usage of various constants.
(1) Within an identity or inequality we will use generic constants $C_i$, $i = 1,2,...$, the different indices are only to emphasize the difference of constants in the same expression. (2) The standard dependence on the fixed physical parameters [ $b>0, c>0, \gamma^{\tau} > 0, k \in \mathbb{R} $]  may not be differentiated. However,  critical dependence on $\tau$ is always emphasized. 

\section{Uniform (in $\tau$) exponential decay  $\Hb^\tau$ -- proof of Theorem \ref{thm1}} \label{unlow}

Let $u^\tau$ be the solution for \eqref{eqnl} whose existence is guaranteed in \cite{jmgt} for every $\tau > 0$. Fix $r_1 > 0$ such that $\|U_0\|_{\Hb^{\tau}} \leqslant r_1.$ 
We shall show that there exists $\rho_1(r_1)$  sufficiently small such that for   $\|U_0\|_{\Ha^{\tau}} \leqslant \rho_1$  we have 
\begin{equation}\label{E1}
 \E^\tau(t) + C_1(r_1,\rho_1) \int_0^t \E^\tau(\sigma)d\sigma \leqslant C_2(r_1) \E^\tau(0)
 \end{equation}
 with $C_1, C_2 > 0, \rho_1 >0$  independent of $\tau$, but possibly  depending  on $r_1$.  Inequality in (\ref{E1}) provides the conclusion of Theorem \ref{thm1} (see \cite{las-tat}). Therefore, our goal is to establish this inequality. The proof  is  accomplished through several steps organized as follows (with the  details given later): \begin{enumerate}
	\item[\bf Step 1:] With reference to the energies \eqref{h0en} and \eqref{h1en}
	we obtain the following $E_1^\tau$-energy identity $$\dfrac{d}{dt}E_1^\tau (t) + \gamma^\tau \|u_{tt}^\tau\|_2^2 = (G(u^\tau),z^\tau_t),$$ where $G(u):= 2k(u^\tau u^\tau_{tt} + (u^\tau_t)^2)$ and $z^\tau \equiv u^\tau_{t} + \dfrac{c^2}{b}u^\tau.$ These calculations are  the same as in (see \cite{jmgt}).
	
	\item[\bf Step 2:] By Sobolev embeddings and interpolation inequalities  we  estimate the  nonlinear terms to obtain 
	that for every $\varepsilon > 0$ there exists constants $C_1(\varepsilon)$ and $C_2 (\varepsilon)$ such that \begin{align}
	E_1^\tau(t)+\int_{0}^{t}&\bigg[\gamma^\tau-(C_1(\varepsilon)\|u^\tau\|_{H^1}^{1/2}||u^{\tau}||_{H^2}^{1/2} +\varepsilon  )\bigg]\|u^\tau_{tt}(\sigma)\|_2^2d\sigma\nonumber\\
	&\leqslant E_1^\tau(0)+\int_{0}^{t}\bigg[C_2(\varepsilon)\|u^\tau_t\|_2\|u^\tau_t\|_{H^1}+\varepsilon \bigg]\|\p u^\tau_{t}(\sigma)\|_2^2d\sigma, \label{st2}
	\end{align}
	
	\item[\bf Step 3:] Combination of Steps 1 and 2  along with another ``energy recovery" estimate  and calibration of the constants   $\varepsilon >0$ gives the first stabilizability inequality:
	 there exist positive constants $C_i, (i = 1,2,3,4)$ such that with  $\gamma^* = \min \{ \gamma^{\tau}, b, c^2  \} $.  $$E^{\tau}(t) + \int_0^t E^{\tau}(\sigma)[C_1 \gamma^* -C_2[E^\tau(\sigma)]^{1/4}\|Au^\tau(\sigma)\|_2^{1/2}-C_3E^\tau(\sigma)]d\sigma \leqslant C_4E^\tau(0).$$

\item[\bf Step 4:] Working further towards the estimate  for the energy $\E^\tau$  we  obtain the  existence of positive constants $\tilde{C}_i, (i = 1,2,3,4)$ such that $$\E^\tau(t) + \int_0^t \E^\tau(\sigma)[\tilde{C}_1 \gamma^*-\tilde{C}_2[E^\tau(\sigma)]^{1/4}[\E^\tau(\sigma)]^{1/4}-\tilde{C}_3E^\tau(\sigma)]d\sigma \leqslant \tilde{C}_4\E^\tau(0).$$

\item[\bf Step 5:] By considering $\E^\tau(0) \leqslant r_1$ and $E^\tau(0) < \rho$  sufficiently small, we obtain  -- from modified Barrier's Method -- \emph{a priori} global bounds. In fact, choosing $\rho$ such that $C_2\rho^{1/4}r_1^{1/4} - C_3\rho \leqslant \gamma^*  C_1/2$ (with reference to the constants from Step 3) and $\tilde{C}_2\rho^{1/4}r_1^{1/4} - \tilde{C}_3\rho \leqslant  \gamma^*  \tilde{C}_1/2$ (with reference to the constants from Step 4) we have the existence of $C>0$ such that the apriori bounds  hold:
\begin{equation}\label{apriori1}
E^\tau(t) \leqslant C\rho \ \ \mbox{and} \ \ \E^\tau(t) \leqslant Cr_1,
\end{equation} uniformly in $t  >0 $ and $\tau \in (0,\tau_0] .$ These  \emph{a priori} bounds imply the desired general inequality
in (\ref{E1}) 
from where the uniform (in $\tau$) exponential decays follows via standard procedure as in, for example, \cite{las-tat,ong}.\end{enumerate}

Next we shall focus on proving  the outlined five steps above. Our analysis is focused on apriori bounds. The issue of existence and regularity  of solutions [for a fixed $\tau$ ] has been dealt with in \cite{jmgt}-thus we can  perform the estimates on  smooth solutions with an eye on obtaining suitable  a priori bounds which are uniform in the parameter $\tau$. .

\begin{proof}[\it Proof of Step 1:]
	 Rewriting  \eqref{eqnl} as
	\begin{equation}\label{nll}
	\tau u^\tau_{ttt}+u^\tau_{tt}+ c^2 Au^\tau + bAu^\tau_t = G(u^\tau),
	\end{equation} where $G(u^\tau) = 2ku^\tau u^\tau_{tt} +2k(u^\tau_t)^2$,
	and   recalling   $E^\tau(t) = E_0^\tau(t) + E_1^\tau(t)$ where (as in \eqref{h0en}), $$E_0^\tau(t)= \frac{1}{2}\|u^\tau_t\|_2^2+ \frac{c^2}{2}\|A^{1/2}u^\tau\|_2^2,$$ with  $E_1^{\tau}$  expanded as \begin{equation}\label{e11}
	E^{\tau}_1(t)=\frac{\tau}{2}\|u^\tau_{tt}\|_2^2+\frac{b}{2}\|\p u^\tau_t\|_2^2 +\frac{c^4}{2b}\|\p u^\tau\|_2^2 + c^2(Au^\tau_t,u^\tau)+ \frac{\tau c^2}{ b} (u^\tau_{tt},u^\tau_t) + \frac{c^2}{2b}\|u^\tau_t\|_2^2,
	\end{equation}
	one obtains,  from the  calculations identical to these in \cite{jmgt}, the following:: \begin{lemma}\label{EI}
		\begin{equation}\label{id1}\dfrac{d}{dt}E_1^\tau (t) + \gamma^\tau \|u_{tt}^\tau\|_2^2 = (G(u^\tau), u_{tt}^{\tau} + c^2 b^{-1} u_t^{\tau} ),
		\end{equation}
	\end{lemma} \end{proof}
	\ifdefined\xxxxxxx
	\begin{proof} We start by taking the $L^2-$ inner product of \eqref{nll} with $u^\tau_{tt}(t) \in L^2(\Omega)$ -- $(t \in [0,T])$ -- and then calling $L(t) = [\mbox{LHS of } \eqref{nll}$] we have, \begin{align}\label{m0}
	(L(t),u^\tau_{tt}(t)) &=  \tau(u^\tau_{ttt},u^\tau_{tt})+(u^\tau_{tt},u^\tau_{tt})+c^2(Au^\tau,u^\tau_{tt})+b(Au^\tau_{t},u^\tau_{tt}) \nonumber \\ &= \dfrac{\tau}{2} \dfrac{d}{dt}\|u^\tau_{tt}\|_2^2 + \|u^\tau_{tt}\|_2^2 + c^2\dfrac{d}{dt}(A u^\tau, u^\tau_{t}) -c^2 (\p u^\tau_t,\p u^\tau_t) + b(\p u^\tau_t,\p u^\tau_{tt}) \nonumber \\ &= \dfrac{\tau}{2} \dfrac{d}{dt}\|u^\tau_{tt}\|_2^2 + \|u^\tau_{tt}\|_2^2 + c^2\dfrac{d}{dt}(A u^\tau, u^\tau_{t}) -c^2 \|\p u^\tau_t\|_2^2 + \dfrac{b}{2}\dfrac{d}{dt}\|\p u^\tau_t\|_2^2 \nonumber \\ &= \dfrac{d}{dt}\left[\dfrac{\tau}{2} \|u^\tau_{tt}\|_2^2 + c^2(A u^\tau, u^\tau_{t})+ \dfrac{b}{2}\|\p u^\tau_t\|_2^2\right] +  \|u^\tau_{tt}\|_2^2 -c^2 \|\p u^\tau_t\|_2^2.
	\end{align}
	
	Next, by taking the $L^2$-inner product of \eqref{nll} with $\dfrac{c^2}{b}u^\tau_t(t) \in \D{\p}$ -- $ t \in [0,T]$ -- we have: \begin{align}\label{m1}
	\dfrac{c^2}{b}(L(t),u^\tau_t(t)) &= \dfrac{c^2}{b}\left\{ \tau(u^\tau_{ttt},u^\tau_{t})+(u^\tau_{tt},u^\tau_{t})+c^2(Au^\tau,u^\tau_{t})+b(Au^\tau_{t},u^\tau_{t})\right\} \nonumber \\ &= \dfrac{c^2}{b}\left\{\tau\dfrac{d}{dt}(u^\tau_{tt},u^\tau_t) - \tau\|u^\tau_{tt}\|_2^2 + \dfrac{1}{2}\dfrac{d}{dt}\|u^\tau_t\|_2^2 + c^2(\p u^\tau,\p u^\tau_t) + b(\p u^\tau_t, \p u^\tau_t)\right\} \nonumber \\ &= \dfrac{c^2}{b}\left\{\tau\dfrac{d}{dt}(u^\tau_{tt},u^\tau_t) - \tau\|u^\tau_{tt}\|_2^2 + \dfrac{1}{2}\dfrac{d}{dt}\|u^\tau_t\|_2^2 + \dfrac{c^2}{2}\dfrac{dt}{dt}\|\p u^\tau\|_2^2 + b\|\p u^\tau_t\|_2^2\right\} \nonumber \\ &= \dfrac{d}{dt}\left[ \dfrac{\tau c^2}{b} (u^\tau_{tt},u^\tau_t) +\dfrac{c^2}{2b}\|u^\tau_t\|_2^2 + \dfrac{c^4}{2b}\|\p u^\tau\|_2^2\right] - \dfrac{\tau c^2}{b}\|u^\tau_{tt}\|_2^2 + c^2\|\p u^\tau_t\|_2^2,
	\end{align} then adding \eqref{m0} and \eqref{m1} we have: \begin{equation}\label{m2}
	(L(t),z^\tau_t(t)) = \dfrac{d}{dt}E^\tau_1(t) + \left(1-\dfrac{\tau c^2}{b}\right)\|u^\tau_{tt}\|_2^2 = \dfrac{d}{dt}E^\tau_1(t) + \gamma^\tau\|u^\tau_{tt}\|_2^2
	\end{equation} and then equating it with the LHS of \eqref{nll} also multiplied (in $L^2$) by $z^\tau(t)$, the identity \eqref{id1} follows.
\end{proof} which was our plan in Step 1.\end{proof}
\fi
\begin{proof}[\it Proof of Step 2:]
	Sobolev Embedding $H^{3/4} (\Omega) \hookrightarrow L^4(\Omega)$ along with interpolation  inequality implies that $\|f\|_{L^4} \leqslant \|f\|_2^{1/4}\|f\|_{H^1}^{3/4},\mbox{for all } f \in H^1(\Omega).$ Moreover, we also have the Sobolev Embedding $H^2(\Omega) \hookrightarrow L^\infty(\Omega)$ 
	with the estimate:$\|u\|_{L^{\infty} } \leqslant C\|u\|^{1/2}_{H^1}\|u\|^{1/2}_{H^2}$
	.This gives:
	 \begin{align}
	\label{m3} (u^\tau u^\tau_{tt},u^\tau_{tt})  &\leqslant C\|u^\tau\|_{L^\infty}\|u^\tau_{tt}\|_2^2\leq C ||u^{\tau}||_{H^1}^{1/2} ||u^{\tau}||_{H^2}^{1/2} ||u_{tt}||^2_{2} 
	\end{align}and  with a given $\varepsilon > 0$, 
 \begin{align}
	\label{m4} (u^\tau u^\tau_{tt},u^\tau_{t})  &\leqslant C_1(\varepsilon)\|u^\tau\|_{L^\infty}^2\|u^\tau_{tt}\|_2^2+\varepsilon C_2\|\p u_{t}\|_2^2 \leq C_1(\epsilon)
	||u^{\tau}||_{H^1}||u^{\tau}||_{H^2} \|u^\tau_{tt}\|_2^2+\varepsilon C_2\|\p u_{t}\|_2^2
	\end{align} Since $\|f^2\|_2^2 = \|f\|_{L^4}^4$ for all $f \in L^4(\Omega)$ we have,\begin{align}\label{m5} ((u^{\tau}_t)^2,u^{\tau}_{tt}) 
	\leqslant C_1(\varepsilon)\|u^\tau_t\|_2\|u^\tau_t\|_{H^1}\|\p u^\tau_t\|_2^2 +\varepsilon\|u^
\tau_{tt}\|_2^2,
	\end{align}  \begin{align}\label{m6}
	((u^\tau_t)^2,u^\tau_{t}) &\leqslant C_1(\varepsilon)\|u^\tau_t\|_2\|u^\tau_t\|_{H^1}\|\p u^\tau_t\|_2^2 +\varepsilon C_2\|\p u^
	\tau_{t}\|_2^2.
	\end{align} 
	Now, \eqref{m3}, \eqref{m4}, \eqref{m5} and \eqref{m6} together imply \begin{align}\label{m7}
	(G(u^\tau),z^\tau_t) 
	\leqslant (C_1(\varepsilon)\|u^\tau\|^{1/2}_{H^1} ||u^{\tau} ||^{1/2}_{H^2} +\varepsilon C_2 )\|u^\tau_{tt}\|_2^2 + (C_3(\varepsilon)\|u^\tau_t\|_2\|u^\tau_t\|_{H^1}+\varepsilon C_4)\|\p u^\tau_t\|_2^2.
	\end{align} 
	
	 Inequalities \eqref{m3} -- \eqref{m7} holds for all $t \in [0,T]$. After integration of  \eqref{id1} from $0$ to $t \in [0,T]$ one  obtains
\begin{align*}
E_1^\tau(t)+\int_{0}^{t}\bigg[\gamma^\tau-(C_1(\varepsilon)\|u^\tau\|_{H^1}^{1/2} ||u^{\tau}||^{1/2}_{H^2} +\varepsilon C_2 )\bigg]&\|u^\tau_{tt}(\sigma)\|_2^2d\sigma\nonumber\\
&\leqslant E_1^\tau(0)+\int_{0}^{t}\bigg[C_3(\varepsilon)\|u^\tau_t\|_2\|u^\tau_t\|_{H^1}+\varepsilon C_4\bigg]\|\p u^\tau_{t}(\sigma)\|_2^2d\sigma,
\end{align*} which reduces to \eqref{st2}  after noting $f \mapsto \|\p f\|_2$ defines a norm in $H_0^1(\Omega)$ -thus 
 completing Step 2.
\end{proof}

\begin{proof}[\it  Proof of Step 3:]
	Here the goal is to obtain  recovery of the  {\it integral}  of  the total energy $E^\tau(t) := E_0^\tau(t) + E_1^\tau(t)$. For this we need to construct $\dfrac{d}{dt}E^\tau_1(t)+\dfrac{d}{dt}E^\tau_0(t).$ 
	\ifdefined\xxxxxxxxxx
	 Going back to equation \eqref{m1} (ignoring the scalar $c^2b^{-1}$) and observe that
	\begin{align}\label{m8} (G(u^\tau),u^\tau_t) &= \dfrac{d}{dt}\left[\tau (u^\tau_{tt},u^\tau_t) +\dfrac{1}{2}\|u^\tau_t\|_2^2 + \dfrac{c^2}{2}\|\p u^\tau\|_2^2\right] - \tau \|u^\tau_{tt}\|_2^2 + b\|\p u^\tau_t\|_2^2 \nonumber \\ & =
	\dfrac{d}{dt}\tau(u^\tau_{tt},u^\tau_t)+\dfrac{d}{dt}E_0^\tau(t)-\tau\|u^\tau_{tt}\|_2^2+b\|\p u^\tau_t\|_2^2,
	\end{align} which combined with \eqref{id1} yelds \begin{align}\label{m9} (G(u^\tau),z^\tau_t + u^\tau_t) &= \dfrac{d}{dt}\tau(u^\tau_{tt},u^\tau_t)+\dfrac{d}{dt}E^\tau_1(t) + \dfrac{d}{dt}E_0^\tau(t)+\left(\gamma^\tau-\tau\right)\|u^\tau_{tt}\|_2^2+b\|\p u^\tau_t\|_2^2 \nonumber \\ &= \dfrac{d}{dt}\tau(u^\tau_{tt},u^\tau_t)+\dfrac{d}{dt}E^\tau(t)+\left(\gamma^\tau-\tau\right)\|u^\tau_{tt}\|_2^2+b\|\p u^\tau_t\|_2^, ,\end{align} then, denoting $\overline{z}^\tau := u^\tau_t + \left(\dfrac{c^2}{b} + 1\right)u^\tau$, 
	\fi
	From direct calculations -- tedious but straightforward and reminisicent to \cite{jmgt} -- we arrive at the identity \begin{equation}
	\label{m10} \dfrac{d}{dt}E^\tau + (\gamma^{\tau}-\tau C_1)\|u^\tau_{tt}\|_2^2 + b\|\p u^\tau_t\|_2^2 = (G(u^\tau),{z}^{\tau}_t +u^{\tau}_t) - \tau\dfrac{d}{dt}(u^\tau_{tt},u^\tau_t). \ \ (C_1 > 0)
	\end{equation}  Integrating inn time from $0$ to $t$, an estimate for the last term on the right hand side of (\ref{m10}) yields:
	\begin{align}\label{m11}
	-\tau\int_0^t \dfrac{d}{dt}(u^\tau_{tt}(\sigma),u^\tau_t(\sigma)d\sigma = -\tau(u^\tau_{tt}(\sigma),u^\tau_t(\sigma))\biggr\rvert_0^t \nonumber \leqslant \dfrac{\tau}{2}\left[\|u^\tau_{tt}(t)\|_2^2 + \|u^\tau_t(t)\|_2^2 + \|u^\tau_{tt}(0)\|_2^2 + \|u^\tau_t(0)\|_2^2 \right] \nonumber \\ = \dfrac{\tau}{2}\|u^\tau_{tt}(t)\|_2^2 + \dfrac{\tau b}{c^2}\dfrac{c^2}{2b}\|u^\tau_t(t)\|_2^2 +\dfrac{\tau}{2}\|u^\tau_{tt}(0)\|_2^2 + \dfrac{\tau b}{c^2}\dfrac{c^2}{2b}\|u^\tau_t(0)\|_2^2 \leqslant \left(1+\dfrac{\tau b}{c^2}\right)(E_1^\tau(t) + E^\tau_1(0)),
	\end{align} where we have used the expansion \eqref{e11}. Combining  \eqref{m11} with the identity \eqref{id1} (to handle $E^\tau(t)$ from \eqref{m11}) and \eqref{m10} to get (after integrating in time): \begin{align}
	E^\tau(t) + (\gamma^{\tau}-\tau C_1) \int_0^t \|u^\tau_{tt}(\sigma)\|_2^2d\sigma &+  b\int_0^t\|\p u^\tau_t(\sigma)\|_2^2d\sigma \leqslant C_2E^\tau(0) \nonumber + C_3\int_0^t|(G(u^\tau(\sigma)),\overline{z}^\tau_t(\sigma)+z^\tau_t(\sigma))|d\sigma,
	\end{align} (here $\overline{z}^\tau := u^\tau_t + \left(\dfrac{c^2}{b} + 1\right)u^\tau$) and after accounting for  \eqref{m7} we obtain the inequality \begin{align}
	\label{m12}E^\tau(t) &+ \int_0^t(\gamma^{\tau}-\tau C_1 -C_2(\varepsilon)\|u^\tau\|_{L^\infty}-\varepsilon C_3 )\|u^\tau_{tt}(\sigma)\|_2^2d\sigma \nonumber \\ &+ \int_0^t(b-C_4(\varepsilon)\|u^\tau_t\|_2\|u^\tau_t\|_{H^1}-\varepsilon C_5)\|\p u^\tau_{t}(\sigma)\|_2^2d\sigma \leqslant C_6E^\tau(0).
	\end{align}
	
To reconstruct the integral $\displaystyle\int_0^t\|\p u^\tau(\sigma)\|_2^2d\sigma$ we take the $L^2$-inner product of \eqref{nll} with $\lambda u^\tau$ ($\lambda>0$ will be chosen later) which gives 
	\begin{align}\label{m13}
	\dfrac{\lambda b}{2}\|\p u^\tau(t)\|_2^2 &+ \lambda c^2\int_0^t \|\p u^\tau(\sigma)\|_2^2d\sigma = \dfrac{\lambda b}{2}\|\p u^\tau(0)\|_2^2+ \lambda\int_0^t \|u^\tau_t(\sigma)\|_2^2d\sigma \nonumber \\ &- \lambda\left[\tau(u_{tt},u^\tau) - \dfrac{\tau}{2}\|u^\tau_t\|_2^2 + (u^\tau_t,u^\tau)\right]\biggr\rvert_0^t + \lambda\int_0^t(G(u^\tau(\sigma)),u^\tau(\sigma))d\sigma.
	\end{align} Denoting by $I = \lambda\displaystyle\int_0^t \|u^\tau_t(\sigma)\|_2^2d\sigma - \lambda\left[\tau(u_{tt},u^\tau) - \dfrac{\tau}{2}\|u^\tau_t\|_2^2 + (u^\tau_t,u^\tau)\right]\biggr\rvert_0^t$ we estimate \begin{align}\label{m14}
	I
	&\leqslant \lambda\left[C_1\int_0^t \|\p u^\tau_t(\sigma)\|_2^2d\sigma +\dfrac{\tau}{2}\|u^\tau_{tt}(t)\|_2^2 + \dfrac{\tau+1}{2}\|u^\tau(t)\|_2^2 + \dfrac{1}{2}\|u^\tau_t(t)\|_2^2 \right] \nonumber\\ &+\lambda\left[ \dfrac{\tau}{2}\|u^\tau_{tt}(0)\|_2^2 + \dfrac{\tau+1}{2}\|u^\tau(0)\|_2^2 + \dfrac{1}{2}\|u^\tau_t(0)\|_2^2 \right] \leqslant \lambda C_1\int_0^t \|\p u^\tau_t(\sigma)\|_2^2d\sigma + \lambda C_2(E^\tau_1(t) + E^\tau_1(0))
	\end{align} which combined with \eqref{m13} gives \begin{align}\label{m15}
	\dfrac{\lambda b}{2}\|\p u^\tau(t)\|_2^2 &+ \lambda c^2\int_0^t \|\p u^\tau(\sigma)\|_2^2d\sigma \leqslant \lambda  C_1[E^\tau(0) +E_1^{\tau} (t) ]  \nonumber \\ &+ \lambda C_2\int_0^t \|\p u^\tau_t(\sigma)\|_2^2d\sigma + \lambda C_3 |\int_0^t (G(u^\tau(\sigma)),u^\tau(\sigma))d\sigma |
	\end{align}
	 It  remains to estimate the nonlinear term $(G(u^\tau), u^\tau)$. For this we have \begin{align}
	\label{m16} (G(u^\tau),u^\tau) &=2k (u^\tau u^\tau_{tt} + (u^\tau_t)^2,u^\tau) \nonumber \\ &\leqslant \varepsilon C_1\|u^\tau\|_2^2 + C_2(\varepsilon)\|u^\tau\|_{L^\infty}^2\|u^\tau_{tt}\|_2^2 + C_3(\varepsilon)\|u^\tau_t\|_2 \|u^\tau_t\|_{H^1}^3 \nonumber \\ & \leqslant \varepsilon C_1\|\p u^\tau\|_2^2 + C_2(\varepsilon)\|u^\tau\|_{L^\infty}^2\|u^\tau_{tt}\|_2^2 + C_3(\varepsilon)\|u^\tau_t\|_2 \|u^\tau_t\|_{H^1}\|\p u^\tau_t\|_2^2,
	\end{align} which holds for all $t \in [0,T].$  Combining  the result of Step 2 and \eqref{m16}  leads to  \begin{align}\label{m17}
	\dfrac{\lambda b}{2}\|\p u^\tau(t)\|_2^2 &+ \lambda (c^2-\varepsilon C_1)\int_0^t \|\p u^\tau(\sigma)\|_2^2d\sigma \leqslant C_2E^\tau(0)\nonumber \\ &+ \lambda\int_0^t(C_3(\varepsilon) \|u^\tau\|_{L^\infty} + C_4(\varepsilon)\|u^\tau\|_{L^\infty}^2 + \varepsilon C_5 )  \|u^\tau_{tt}(\sigma)\|_2^2 d\sigma \nonumber \\ & + \lambda\int_0^t(C_6(\varepsilon)\|u^\tau_t\|_2\|u^\tau_t\|_{H^1} + \varepsilon C_7 + C_8)\|\p u^\tau_t(\sigma)\|_2^2d\sigma.
	\end{align} Then, combining \eqref{m12} with \eqref{m17} 
	and  picking $0 < \lambda << \dfrac{b}{2}$, $\tau_0$ and $\varepsilon$ sufficiently small (and fixed) we obtain 
	 \begin{align}\label{m19}
	E^\tau(t) + \int_0^t(\gamma^{\tau}  -C_2\|u^\tau\|_{L^\infty}) \|u^\tau_{tt}(\sigma)\|_2^2 d\sigma  &+ \int_0^t(C_3b -C_4\|u^\tau_t\|_2\|u^\tau_t\|_{H^1})\|\p u^\tau_{t}(\sigma)\|_2^2d\sigma \nonumber \\ &+C_5 c^2 \int_0^t \|\p u^\tau(\sigma)\|_2^2d\sigma \leqslant C_{6}E^\tau(0).
	\end{align} 
	Applying once more interpolation inequality 
	$\|u^\tau\|_{L^\infty} \leqslant C\|u^\tau\|_{H^2}^{1/2} \|u^\tau\|_{H^1}^{1/2}$
	 leads to 
	\ifdefined\xxxxxxx
	With this in hand, we rewrite \eqref{m19} as \begin{align*}
	E^\tau(t) &+ C_1\int_0^t E^\tau(\sigma)d\sigma - C_2\int_0^t\|u^\tau\|_{H^2}^{1/2} \|u^\tau\|_{H^1}^{1/2} \|u^\tau_{tt}(\sigma)\|_2^2 d\sigma \nonumber \\ &-C_3\int_0^t\|u^\tau_t\|_2\|u^\tau_t\|_{H^1}\|\p u^\tau_{t}(\sigma)\|_2^2d\sigma \leqslant C_{4}E^\tau(0)
	\end{align*} and estimating above by $E^\tau(t)$ we arrive at \begin{align*}
	E^\tau(t) &+ C_1\int_0^t E^\tau(\sigma)d\sigma - C_2\int_0^t \|Au^\tau(\sigma)\|_{2}^{1/2} [E^\tau(\sigma)]^{5/4} d\sigma \nonumber \\ &-C_3\int_0^t[E^\tau(\sigma)]^{2}d\sigma \leqslant C_{4}E^\tau(0)
	\end{align*} or further
	\fi 
	 \begin{align}\label{m20}
	E^\tau(t) &+ \int_0^t E^\tau(\sigma)[\gamma^{*} - C_2\|Au^\tau(\sigma)\|_2^{1/2}[E^\tau(\sigma)]^{1/4} - C_3E^\tau(\sigma)]d\sigma \leqslant C_{4}E^\tau(0),
	\end{align} and here all the constants are positive and independent of $\tau$ and on time and $\gamma^* =\min \{ \gamma^{\tau} ,  Cb,C c^2 \} $ where $C>0$ is a suitable generic constant. This completes the proof of Step 3.
\end{proof}
\ifdefined\xxxxxxxxx
\begin{rmk}
	\label{interp} The inequality \eqref{intineq} was used here as a special case of Theorem 5.9 (see \cite{adam}, p. 141 and 142), where the general inequality is given as \begin{equation}\label{adin}\|u\|_{L^\infty} \leqslant C \|u\|_{W^{m,p}}^\theta \|u\|_{L^q}^{1-\theta},\end{equation} where either $1 \leqslant q \leqslant p$ or both $q>p$ and $p(m-1) < n$ and $\theta$ (which depends on $n,p,q$ and $m$) is given by $\theta = \dfrac{np}{np+(mp-n)q}.$
	
	Notice that, algebraically speaking, our case fits the setting $m = 2, p=2, q = 2$ and $n = 2$ or $3.$ This would imply $\theta = \dfrac{n}{4}$ and the inequality \eqref{adin} would read $$\|u\|_{L^\infty} \leqslant C \|u\|_{H^2}^{n/4} \|u\|_{2}^{(4-n)/4} \leqslant C\|u\|_{H^2}^{n/4} \|u\|_{H^1}^{(4-n)/4}$$ which is exactly \eqref{intineq} for $n =2$, however, the case $n = 3$ \emph{only} works because $u^\tau$ is a solution of our nonlinear problem and, therefore, exists in a small ball which we assume to have radius less than one ($\rho < 1$). Then we can take $\theta = 1/2$ for both cases $n=2$ and $n=3$.
\end{rmk}
\fi
\begin{proof}[\it Proof of Step 4:] Our objective in this step is to establish a stabilizabity inequality for the energy $\E^\tau.$  Since 
$\E^\tau(t) \approx E^\tau(t) + \|Au^\tau(t)\|_2^2$,  it suffices to focus on the higher order term $\|Au^\tau(t)\|_2^2.$ This quantity can be estimated using the multiplier $Au^\tau \in L^2(\Omega)$. Taking the $L^2$-inner product of \eqref{nll} with $\lambda Au^\tau$ $(\lambda > 0$ to be later determined) we have, \begin{align*}
	\lambda(G(u^\tau),Au^\tau) 
	 = \lambda \left\{\dfrac{d}{dt}\left[\tau (u^\tau_{tt},Au^\tau) - \dfrac{\tau}{2}\|\p u^\tau_t\|_2^2 + (u^\tau_{t},Au^\tau)\right]- \|\p u^\tau_t\|_2^2 + \dfrac{b}{2}\dfrac{d}{dt}\|Au^\tau\|_2^2 + c^2 \|Au^\tau\|_2^2 \right\},
	\end{align*} then, integration in time from $0$ to $t \in [0,T]$ gives \begin{align}
	\label{m22} \dfrac{\lambda b}{2}\|Au^\tau\|_2^2 &+ \lambda c^2 \int_0^t\|Au^\tau(\sigma)\|_2^2d\sigma = \dfrac{\lambda b}{2}\|Au^\tau(0)\|_2^2 + \lambda \int_0^t \|\p u^\tau_t(\sigma)\|_2^2 d\sigma  \nonumber \\&-\lambda \left[\tau (u^\tau_{tt},Au^\tau) - \dfrac{\tau}{2}\|\p u^\tau_t\|_2^2 + (u^\tau_{t},Au^\tau)\right]\biggr\rvert_0^t + \lambda \int_0^t (G(u^\tau(\sigma)),Au^\tau(\sigma))d\sigma,
	\end{align} Denoting  $I = \left[\tau (u^\tau_{tt},Au^\tau) - \dfrac{\tau}{2}\|\p u^\tau_t\|_2^2 + (u^\tau_{t},Au^\tau)\right]\biggr\rvert_0^t $ we have \begin{align}
	I 
	&\leqslant \dfrac{\tau}{2}\|u^\tau_{tt}(t)\|_2^2 + \dfrac{\tau }{2}\|Au^\tau(t)\|_2^2 + \dfrac{\tau+1}{2}\|\p u^\tau_t(t)\|_2^2 + \dfrac{1}{2}\|\p u^\tau(t)\|_2^2 \nonumber \\ &+ \dfrac{\tau}{2}\|u^\tau_{tt}(0)\|_2^2 + \dfrac{\tau }{2}\|Au^\tau(0)\|_2^2 + \dfrac{\tau+1}{2}\|\p u^\tau_t(0)\|_2^2 + \dfrac{1}{2}\|\p u^\tau(0)\|_2^2 \nonumber \leqslant C_1\E^\tau(0) + \dfrac{\tau}{2}\|Au^\tau(t)\|_2^2 +  C_2E^{\tau} (t) 
	\end{align} 
	Then from \eqref{m22} it follows that \begin{align}
	\label{m23} \dfrac{\lambda (b-\tau_0)}{2}\|Au^\tau (t) \|_2^2 &+ \lambda c^2 \int_0^t\|Au^\tau(\sigma)\|_2^2d\sigma \leqslant C_1(\lambda)\E^\tau(0)+  C_2\lambda 
	E^{\tau}(t) \nonumber \\& + \lambda C_3\int_0^t \|\p u^\tau_t(\sigma)\|_2^2 d\sigma + \lambda \int_0^t (G(u^\tau(\sigma)),Au^\tau(\sigma))d\sigma.
	\end{align} For  the nonlinear part we have \begin{align}\label{m24}
	(G(u^\tau),Au^\tau) &\leqslant \varepsilon C_1\|Au^\tau\|_2^2 + C_2(\varepsilon)\|u^\tau\|_{L^\infty}^2\|u^\tau_{tt}\|_2^2 + C_3(\varepsilon)\|u^\tau_t\|_2 \|u^\tau_t\|_{H^1}^3 \nonumber \\ & \leqslant \varepsilon C_1\|A u^\tau\|_2^2 + C_2(\varepsilon)\|u^\tau\|_{L^\infty}^2\|u^\tau_{tt}\|_2^2 + C_3(\varepsilon)\|u^\tau_t\|_2 \|u^\tau_t\|_{H^1}\|\p u^\tau_t\|_2^2 
	\end{align} then a combination of \eqref{m23}, \eqref{m24} and \eqref{m19} gives \begin{align}\label{m25}
	E^\tau(t) + \dfrac{\lambda b}{4} ||Au^{\tau}(t)||^2_2  &+ \lambda(c^2  -\varepsilon C_1)\int_0^t \|Au^\tau(\sigma)\|_2^2d\sigma \nonumber \\ & +\int_0^t(C_2-(C_3+\lambda C_4(\varepsilon))\|u^\tau\|_{L^\infty}) \|u^\tau_{tt}(\sigma)\|_2^2 d\sigma \nonumber \\ &+ \int_0^t(C_5-(C_6+\lambda C_7(\varepsilon))\|u^\tau_t\|_2\|u^\tau_t\|_{H^1}-\lambda C_8))\|\p u^\tau_{t}(\sigma)\|_2^2d\sigma \nonumber \\ &+ C_9\int_0^t \|\p u^\tau(\sigma)\|_2^2d\sigma \leqslant C_{10}E^\tau(0) + C_{11}\lambda E^{\tau}(t) 
	\end{align} then, first fixing $\varepsilon > 0$ small and then fixing $\lambda << C_5/C_8$ it follows that \begin{align}\label{m26}
	\E^\tau(t) &+ \int_0^t(\gamma^\tau-C_1\|u^\tau\|_{L^\infty}) \|u^\tau_{tt}(\sigma)\|_2^2 d\sigma + \int_0^t(C_2-C_3\|u^\tau_t\|_2\|u^\tau_t\|_{H^1})\|\p u^\tau_{t}(\sigma)\|_2^2d\sigma \nonumber \\ &+ C_4\int_0^t \|\p u^\tau(\sigma)\|_2^2d\sigma + C_5\int_0^t\|Au^\tau(\sigma)\|_2^2 \leqslant C_{6}\E^\tau(0),
	\end{align} which reduces to   
	\ifdefined\xxxx
	 \begin{align}\label{m27}
	\E^\tau(t) &+ C_1 \int_0^t \E^\tau(\sigma)d\sigma -C_2\int_0^t\|u^\tau\|_{L^\infty} \|u^\tau_{tt}(\sigma)\|_2^2 d\sigma -C_3 \int_0^t\|u^\tau_t\|_2\|u^\tau_t\|_{H^1}\|\p u^\tau_{t}(\sigma)\|_2^2d\sigma \leqslant C_{4}\E^\tau(0),
	\end{align} and further, by \eqref{intineq} we have 
	\fi
	 \begin{align}\label{m28}
	\E^\tau(t) &+ \int_0^t \E^\tau(\sigma)[\gamma^* - C_1[\E^\tau(\sigma)]^{1/4}[E^\tau(\sigma)]^{1/4} - C_2E^\tau(\sigma)]d\sigma \leqslant C_{3}\E^\tau(0),
	\end{align} with  all the constants being positive and independent of $\tau$ and on time. This completes the proof of Step 4.\end{proof}

\ifdefined\xxxxxx
\begin{proof}[\bf Proof of Step 5:]
	Observe now that our assumption $\E^\tau(0) \leqslant r_1$ with $r_1$ a fixed positive number and $E^\tau(0) < \rho$, $\rho >0$ small gives -- via \eqref{m28} along with a slighly modified Barrier's Method argument -- global boundedness of $\E^\tau(t)$ and global smallness of $E^\tau(t)$. The argument is sketched as follows: with $t = 0$ and fixed $r_1$ we pick $\rho$ so the coefficient of $\E^\tau(\sigma)$ inside the integral in -- both -- \eqref{m28} and \eqref{m20} is such that $$C_1 - C_2[\E^\tau(0)]^{1/4}[E^\tau(0)]^{1/4} - C_3E^\tau(0) > C_1 - C_2r_1^{1/4}\rho^{1/4} - C_3\rho > \dfrac{C_1}{2}.$$ When we let nature take its course, however, our bounds would be assured as long as we have the positivity $C_1 - C_2[\E^\tau(\sigma)]^{1/4}[E^\tau(\sigma)]^{1/4} - C_3E^\tau(\sigma) > 0$ for all $0 < \sigma \leqslant t.$ Since we pick $\rho$ small so it \emph{starts} positive and both $t \mapsto \E^\tau(t)$ and $t \mapsto E^\tau(t)$ are continuous, we just need to guarantee that there is no $T > 0$ such that $C_1 - C_2[\E^\tau(T)]^{1/4}[E^\tau(T)]^{1/4} - C_3E^\tau(T) = 0.$ Suppose, by contradition, that this is not the case and let's consider the smallest such $T$. In that case, it follows from \eqref{m28} and \eqref{m20} that $C_1-C_2[\E^\tau(\sigma)]^{1/4}[E^\tau(\sigma)]^{1/4} + C_3E^\tau(\sigma)$ can be made greater than or equals to $C_1/2$ for all $0 \leqslant \sigma < T$ in a manner which is not dependent of $T$, which would imply that the function $t \mapsto C_1-C_2[\E^\tau(t)]^{1/4}[E^\tau(t)]^{1/4} + C_3E^\tau(t)$ has a jump discontinuity at $t = T$, which is a contradiction. Then, there exists $C$ such that $$E^\tau(t) \leqslant C\rho \ \ \mbox{and} \ \ \E^\tau(t) \leqslant Cr_1$$ uniformly in $t \in (0,\infty)$ and $\tau > 0.$ Therefore, the inequality \begin{align}\label{m29}
	 \E^\tau(t) &+ C_1\int_0^t \E^\tau(\sigma)d\sigma \leqslant C_2\E^\tau(0),
	 \end{align} holds for all $t \in (0,\infty).$
	 \fi
	 \begin{rmk}\label{imp}
	 	Note that we also obtain the estimate 
	 	\begin{equation}\label{lower} 
	 	E^{\tau} (t) + C_1(r_1,\rho_1)\int_0^t E^{\tau}(s) ds \leq C_2(r_1,\rho_1)  E^{\tau}(0) 
	 	\end{equation}
	 	for all solutions $u^{\tau} $ such that $\E^\tau(0) \leq r_1 $ and  $E^\tau(0) \leq \rho_1(r_1)  $ with $\rho_1$ sufficiently small. 
	 	This is to say that under the smallnes condition imposed on the lowest energy $E^\tau(0)$ -- which depends on the bound of higher energy $\E^\tau(0)$ -- the lower energy of the solutions -- as a trajectory -- remains bounded  by a multiple of $E^\tau(0)$ for all times. This fact will be used later several times. \end{rmk}

\section{Uniform (in $\tau$) exponential decay in $\Hc^\tau.$ -the proof of Theorem \ref{thm2}} \label{unhigh}

With an eye on higher topology of $\Hc$ space, we shall repeat the procedure of the previous section. 
We will now show that the solution for \eqref{eqnl}  is uniform (in $\tau$) exponentially stable in the topology of $\Hc^\tau$ -- recall that $\Hc = \D{A} \times \D{A} \times \D{\p}$ -- under the   smallness condition in $\mathbb{H}_0^{\tau}  $ only.  For $\tau > 0$, let $u^\tau$ be such solution and fix $r_2 > 0$ such that $\|U_0\|_{\Hc} \leqslant r_2$ and let $\|U_0\|_{\Ha^{\tau}} \leqslant \rho_2$, $\rho =\rho(r_2) $ sufficiently small  will be determined in the course of the proof.  Again, for the energy functional $\He^\tau$ defined as $\He^\tau(t) \approx \E^\tau(t) + \|Au^\tau_t\|_2^2 + \tau \|\p u^\tau_{tt}\|_2^2$, we seek to establish the general stabilizability inequality \begin{equation}\label{sth2}\He^\tau(t) + C_0  \int_0^t ||A^{1/2} u_{tt}||_2^2  d\sigma  + C_1\int_0^t \He^\tau(\sigma)d\sigma \leqslant C_2\He^\tau(0),\end{equation} with $C_0,C_1, C_2 > 0$ and independent of $\tau$ but they depend on $r_2$ and $\rho_2$.  To accomplish this 
one needs to  account for  higher order terms $\|Au^\tau_t\|_2^2$ and $\|A^{1/2}u_{tt}^{\tau}\|$. 
The procedure is similar as in the previous case  (Steps 1-5) of $\mathbb{H}_1$ topology, so we shall concentrate only on the details which are different and require adititonal care. 
 Recall that we are working with  sufficiently smooth solutions guaranteed by the wellposedness and regularity theory, so our computations ahead can be rigorously justified. 

\ifdefined\xxxxx
\begin{itemize}
	\item[\bf Step 1:] We use the multipliers (or some multiple of) $Au^\tau_t$ and $Au^\tau_{tt}$ in order to reconstruct the high order terms $\|Au^\tau_t\|_2^2$ and $\|Au^\tau_t\|$

	\item[\bf Step 2:]
	
	\item[\bf Step 3:]
\end{itemize}
\fi

\begin{proof}[\bf Proof of (\ref{sth2}):] We take the $L^2$-inner product of \eqref{nll} with $Au^\tau_t(t) \in L^2(\Omega)$ for all $t \in [0,T]$. Recalling that $G(u^\tau) = 2k(u^\tau u^\tau_{tt} + (u^\tau_t)^2)$ we obtain, 
	\begin{align}\label{m31}
	\dfrac{c^2}{2}\|Au^\tau\|_2^2&+b\int_{0}^{t}\|Au^\tau_{t}(\sigma)\|_2^2d\sigma-\tau\int_{0}^{t}\|\p u^\tau_{tt}(\sigma)\|^2d\sigma\nonumber\\
	&=\dfrac{c^2}{2}\|Au^\tau(0)\|^2-\left[\tau(u^\tau_{tt},Au^\tau_{t})+\dfrac{1}{2}\|\p u^\tau_t\|_2^2\right]\biggr\rvert_0^t+\int_{0}^{t}(G(u^\tau(\sigma)),Au^\tau_t(\sigma))d\sigma.
	\end{align}

The next step is to take the $L^2$-inner product of \eqref{nll} with $\lambda Au^\tau_{tt}(t) \in L^2(\Omega)$ for all $t \in [0,T]$ (smooth solutions) where $\lambda>0$ will be determined later. This leads to, 
	\begin{align}\label{m33}\dfrac{\lambda\tau}{2}\|\p u^\tau_{tt}\|_2^2&+\dfrac{\lambda b}{2}\|Au^\tau_t\|_2^2
-\lambda c^2\int_0^t\|A u^\tau_t(\sigma)\|_2^2d\sigma + \lambda\int_0^t\|\p u^\tau_{tt}\|_2^2d\sigma = \dfrac{\lambda\tau}{2}\|\p u^\tau_{tt}(0)\|_2^2\nonumber\\
&+\dfrac{\lambda b}{2}\|Au^\tau_t(0)\|_2^2 -\lambda c^2\left[(Au^\tau,Au^\tau_t)\right]\biggr\rvert_0^t+ \lambda\int_0^t(G(u^\tau(\sigma)),Au^\tau_{tt}(\sigma))d\sigma
\end{align} and then adding $\eqref{m31}$ with $\eqref{m33}$ we arrive at \begin{align}\label{m34}
\dfrac{c^2}{2}\|Au^\tau\|_2^2&+\dfrac{\lambda\tau}{2}\|\p u^\tau_{tt}\|_2^2+\dfrac{\lambda b}{2}\|Au^\tau_t\|_2^2+(b-\lambda c^2) \int_{0}^{t}\|Au^\tau_{t}(\sigma)\|_2^2d\sigma \nonumber \\ &+(\lambda-\tau)\int_{0}^{t}\|\p u^\tau_{tt}(\sigma)\|^2d\sigma=\dfrac{c^2}{2}\|Au^\tau(0)\|^2+\dfrac{\lambda\tau}{2}\|\p u^\tau_{tt}(0)\|_2^2+\dfrac{\lambda b}{2}\|Au^\tau_t(0)\|_2^2\nonumber \\ &-\left[\tau(u^\tau_{tt},Au^\tau_{t})+\dfrac{1}{2}\|\p u^\tau_t\|_2^2+\lambda c^2 (Au^\tau,Au^\tau_t)\right]\biggr\rvert_0^t+\int_{0}^{t}(G(u^\tau(\sigma)),Au^\tau_t(\sigma)+\lambda Au^\tau_{tt}(\sigma))d\sigma.
\end{align}

Now, letting $I = \left[\tau(u^\tau_{tt},Au^\tau_{t})+\dfrac{1}{2}\|\p u^\tau_t\|_2^2+\lambda c^2 (Au^\tau,Au^\tau_t)\right]\biggr\rvert_0^t$ we estimate \begin{align}\label{m35}
I
  &\leqslant C_1(\lambda)\He^\tau(0) + \dfrac{\tau}{\lambda b}\|u^\tau_{tt}(t)\|_2^2+\dfrac{(\tau +1)\lambda b}{4}\|Au^\tau_t(t)\|_2^2 +\dfrac{1}{2}\|A^{1/2}u^\tau_t(t)\|_2^2 + \dfrac{\lambda c^4}{b}\|Au^\tau(t)\|_2^2 \nonumber  \\ &\leqslant C_1(\lambda)\He^\tau(0) + C_2(\lambda)\E^\tau(0) +\dfrac{(\tau +1)\lambda b}{4}\|Au^\tau_t(t)\|_2^2 \leqslant C_1(\lambda)\He^\tau(0) +\dfrac{(\tau +1)\lambda b}{4}\|Au^\tau_t(t)\|_2^2,
\end{align} For the nonlinear terms we have :   for all $\varepsilon > 0$ there exists constants $C_i(\epsilon) , i =1,2,3 $  such that  \begin{align}
\label{m36} (G(u^\tau),Au^\tau_t) &= (u^\tau u^\tau_{tt} + (u^\tau_t)^2,Au^\tau_t) \nonumber \\& \leqslant \varepsilon C_1 \|Au^\tau_t\|_2^2 + C_2(\varepsilon)\|u^\tau\|_{L^\infty}^2\|\p u^\tau_{tt}\|_2^2 + C_3(\varepsilon)\|u_t^\tau\|_2\|u^\tau_t\|_{H^1}\|A u^\tau_t\|_2^2,
\end{align} and  by  Poincaré's inequality, Sobolev's embeddings, interpolation inequalities and product rule   we obtain \begin{align}
\label{m37} (G(u^\tau),Au^\tau_{tt}) &= (u^\tau u^\tau_{tt} + (u^\tau_t)^2,Au^\tau_{tt}) = \left(A^{1/2}(u^\tau u^\tau_{tt} + (u^\tau_t)^2), \p u^{\tau}_{tt}\right) 
 \leq   ||  \nabla(u^\tau_{tt} u^\tau + (u_t^{\tau})^2 )||_2 || \p u^\tau_{tt}\||_2 \nonumber \\& \leqslant C_1(\varepsilon)\|u^\tau_{tt}\|_{L^4}^2\|\p u^\tau\|_{L^4}^2 + (\varepsilon + \|u^\tau\|_{L^\infty})\|\p u^\tau_{tt}\|_2^2 + C_2(\varepsilon)\|u^\tau\|_{L^4}^2\|\p u^\tau_t\|_{L^4}^2 \nonumber \\& \leqslant C_1(\varepsilon)\|\p u^\tau_{tt}\|_{2}^2\|\p u^\tau\|_2^{1/2}\|Au^\tau\|_{2}^{3/2} + (\varepsilon + \|u^\tau\|_{L^\infty})\|\p u^\tau_{tt}\|_2^2 + C_2(\varepsilon)\|\p u^\tau\|_{2}^2\|A u^\tau_t\|_{2}^2 \nonumber \\ & \leqslant \left(\varepsilon + \|u^\tau\|_{L^\infty} + C_1(\varepsilon)\|\p u^\tau\|_2^{1/2}\|Au^\tau\|_{2}^{3/2}\right)\|\p u^\tau_{tt}\|_2^2 + C_2(\varepsilon)\|\p u^\tau\|_{2}^2\|A u^\tau_t\|_{2}^2 
\end{align}  We then combine \eqref{m34}, \eqref{m35}, \eqref{m36} and \eqref{m37} we get \begin{align}\label{m38}
\dfrac{c^2}{2}\|Au^\tau\|_2^2&+\dfrac{\lambda\tau}{2}\|\p u^\tau_{tt}\|_2^2+\dfrac{\lambda(1-\tau)}{4}\|Au^\tau_t\|_2^2+\int_{0}^{t}(b-\lambda C_1-\varepsilon \lambda C_2)\|Au^\tau_{t}(\sigma)\|_2^2d\sigma \nonumber \\ &+\int_{0}^{t}(\lambda-\tau-\varepsilon \lambda C_3)\|\p u^\tau_{tt}(\sigma)\|^2d\sigma \leqslant C_4(\lambda)\He^\tau(0) \nonumber \\ &+ \lambda\int_0^t C_5(\varepsilon)\left(\|u^\tau\|_{L^\infty}^2+\|\p u^\tau\|_2^{1/2}\|Au^\tau\|_{2}^{3/2}\right)\|\p u^\tau_{tt}(\sigma)\|_2^2d\sigma + \lambda \int_0^tC_6(\varepsilon)\|u_t^\tau\|_2\|u^\tau_t\|_{H^1}\|A u^\tau_t(\sigma)\|_2^2d\sigma,
\end{align} and by making $\tau_0, \lambda < \dfrac{b}{2C_1} $ and $\varepsilon$ small enough, using the interpolation inequality, and taking into account inequality  \eqref{m19} we arrive at \begin{align*}
\He^\tau(t) &+   \int_0^t \frac{1}{8}  b ||A^{1/2}u_{tt}|(\sigma ) ||^2 d \sigma  +  \gamma^{**} \int_0^t \He^\tau(\sigma)d\sigma - C_1\int_0^t \|u^\tau\|_{H^2}^{1/2}\|u^\tau\|_{H^1}^{1/2}\|\p u^\tau_{tt}(\sigma)\|_2^2d\sigma \nonumber \\&- C_2\int_0^t\|u_t^\tau\|_2\|u^\tau_t\|_{H^1}\|A u^\tau_t(\sigma)\|_2^2d\sigma \leqslant C_3\He^\tau(0)
\end{align*} which can be rewritten in terms of the energy functionals  as:
\begin{align}
\He^\tau(t) &+  \int_0^t ||A^{1/2} u_{tt} (\sigma) ||^2\left[\frac{1}{8} b -[E^{\tau}(\sigma) ]^{1/4}[\E^{\tau}(\sigma)]^{1/4}\right]   +\int_0^t \He^\tau(\sigma)(\gamma^{**} - C_1[E^\tau(\sigma)]^{1/2}[\E^\tau(\sigma)]^{1/2})d\sigma \leqslant C_2\He^\tau(0), \label{h2in}
\end{align}
where $\gamma^{**} = \min \{ \gamma^{*}, b\} $.  Barrier's Method  applied  to the last inequality  asserts global boundedness   of $\He^\tau$ by a multiple of $r_2$, in a similar manner as in the previous section. This  leads to  the final estimate \eqref{sth2}. The proof is complete. \end{proof}
\begin{rmk}
 Notice that the parameters responsible for $\mathbb{H}_i$ stability estimates  are $\gamma^\tau > 0 $, $b >0, c^2 > 0 $. 
\end{rmk}

\section{ Strong convergence and convergence  rate  -proof of Theorem \ref{thm3}}\label{conv}

Our aim is  to establish strong convergence of the flows $U^{\tau}(t,U_0) $ , when $\tau \rightarrow 0 $ to  a solution $ U^0(t, P U_0)$  of Westervelt equation  \eqref{jmgt0} with initial data $PU_0 \in P(\mathbb{H}_1) $.   By strong we mean in the  strong topology of the phase  space $\mathbb{H}_1$. 
The argument will follow through the following two  steps. 
In the first one we shall derive convergence rates -- uniform convergence -- for initial data in  more regular space $\Hc$ and on  the  finite time horizon. In the second part we shall prove strong convergence in the phase space of dynamical system and for  the initial data also taken from the phase space $\mathbb{H}_1$. In this context we wish to emphasize the following difficulty: standard argument used for this type of results is based on consistency, stability and density; 
However, in the case of quasilinear problems the initial data under consideration are required to be sufficiently small, which does not cooperate with 
the usual density argument. In order to deal with the issue a careful analysis of topological smallness is necessary. For this reasons the results on uniform exponential decays of solutions obtained in previous sections  differentiate between the topology where the smallness and boundedness  is required. The appropriate calibration  of "smallness" and regularity is critical for the argument. 

To begin with let us denote  $x^\tau:= u^\tau - u^0$ where $u^{\tau} $ is a solution of \eqref{eqnl}  and $u^0 $  the solution  of Westervelt equation  \eqref{jmgt0}.We will prove:

 \begin{theorem}\label{cc0}
 For each arbitrary $T>0$ and $M>0$ there exists $\rho_M >0$, sufficiently small, and an increasing continuous function $K_T: \mathbb{R}^+ \to \mathbb{R}^+$ such that the following inequality \begin{equation}
	\label{g0} \|(x^\tau(t),x_t^\tau(t))\|_{\D{A}\times\D{\p}}^2 = \|Ax^\tau(t)\|_2^2 + \|\p x^\tau_t(t)\|_2^2 \leqslant \tau K_T(||U_0||^2_{\mathbb{H}_2}),
	\end{equation} is true for  all initial data such that $ \|U_0\|_{\mathbb{H}_2}  \leq M $ and $E^{\tau} (0) \leq \rho_M $,    $\tau \in \Lambda = (0,\tau_0]$ and all $t \in [0,T].$ 
\end{theorem}

\begin{rmk}
	Notice that part (a) of Theorem \ref{thm3}  -- although written in the notation of the  nonlinear flows  -- is equivalent to the first statement in  Theorem \eqref{cc0}.
\end{rmk}
Clearly, the regularity of $u^0$ is going to play  the major role. The linear part of Westervelt equation generates an analytic semigroup and the system has ``maximal regularity" \cite{hieber,lunardi}. These are powerful tools in the study of the regularity of \eqref{jmgt0}. For instance, \cite{wilke} provides a refined  theory of wellposedness 
for \eqref{jmgt0} within the  $L^p$ framework. As to the  Hilbert framework (relevant to this work)  - from Theorem 1.1 in \cite{WE} we know that 
for the initial data $\|Au(0)\|+ \|Au_t(0)\| $ sufficiently small one obtains unique  global solution in  $ u^0 \in C^s([0,T], H^{2-s}(\Omega) ) , s =0,1,2 $. However,the  above results will not be sufficient for our analysis. We will need a tighter control of the smallness and regularity -- a result stated later  in Theorem \ref{t0}. 

The proof of Theorem \ref{cc0}  will be accomplished through the following  steps: 
	We first prove that $x^\tau = u^\tau-u^0$ satisfy a (second order in time) PDE with forcing term dependent of $x$ and $u$ as well on their derivatives. Applying the multipliers $Ax^\tau_t$ and $x^\tau_{tt}$  reconstructs the terms on the left side of \eqref{g0}. However, the right hand side becomes singular. Handling this singularity is the main component of the proof. We are faced with the following challenges:
		({\bf{1})} Can the integral term $\tau\displaystyle\int_0^t \|u^\tau_{ttt}(\sigma)\|_2^2d\sigma$ be controlled by the $\Hc$  initial data which are small in $\mathbb{H}_0$? {({\bf 2})}  Can we work with the limit evolution $U^0(t) $  under less stringent regularity and smallness assumptions? 
	Positive answers to these questions are given below.  In fact, from the result of Theorem \ref{wph2} we already know that for each $\tau > 0 $, $u_{ttt}^{\tau}(t) 
	$ is in $L_2(\Omega) $. The issue is that of controlling it's singularity when $ \tau \rightarrow 0 $. This is  the content of the Lemma below. 
	  \begin{lemma}
		\label{g1}  Let $U^{\tau} $ be a solution of (\ref{eqnl}). Assume $||U_0||_{\mathbb{H}_2} \leq M$   and $\|U_0\|_{\mathbb{H}_0} < \rho_M$ with  $\rho_M$  small. Then, there exists a continuous bounded function  $C(s)>0 $,  $($independent of $\tau \in \Lambda $$)$, such that $$\int_0^t \|u^\tau_{ttt}(\sigma)\|_2^2 d\sigma  \leqslant  \frac{1}{\tau} C(\|U_0\|_{\Hc}^2).$$
	\end{lemma}
	
	\begin{theorem}\label{t0} Let $U^0 =(u^0, u_t^0) $ be a solution to the Westervelt equation with the initial conditions subject to the following assumption: 
For each $M>0$ there exists $\rho_M$ small such that  $$\dfrac{1}{2}[\|u^0_t(0)\|_2^2+\|A^{1/2}u^0(0)\|_2^2] \leqslant \rho_M \ \ \mbox{and} \ \ \dfrac{1}{2}[\|\p u^0_t(0)\|^2+\|Au^0(0)\|^2] \leqslant M.$$ Then, there exists constants $C(\rho_M,M)$ and $\omega_0 > 0 $ such that 
		$$\|Au^0(t)\|_2^2 + \|A^{1/2}u^0_t(t)\|_2^2 \leq  C (M)  e^{-\omega_0 t} , ||A^{1/2}u^0(t)\|_2^2 + \|u^0_t(t)\|_2^2 \leq  \rho_M  C(M)  e^{-\omega_0 t}$$
		\begin{equation}\label{lime}\int_0^t [\|u^0_{tt}(s)\|_2^2  + \|Au^0(s)\|_2^2 + \|A^{1/2}u^0_t(s)\|_2^2]ds  \leq  C (\rho,M), t >0 \end{equation}
	\end{theorem}



\begin{proof}[\bf Proof of Lemma \ref{g1}:]
Since $\A^{\tau} $ generates a linear semigroup on  each of the spaces $\mathbb{H}_i$, $i=0,1,2 $, \cite{bongarti,mgtp1,TrigMGT},  linear semigroup theory allows to  represent any solution of the non-homogenous problem  
$U^{\tau}_t = \A^{\tau} U^{\tau} (t) + F(t)  $ with $F\in L^1(0,T, \mathbb{H}_i) )$,   via  the variation of parameters formula: 
$U^\tau(t) = e^{ \A^\tau  t} U_0 + \int_0^te ^{ \A^\tau  (t-\sigma) }  F(t-\sigma) d\sigma  $ . When $F \in W^{1,1}((0,T), \mathbb{H}_i)$  and $\A^{\tau} U_0 - F(0)\in \mathbb{H}_i $,   one also has 
$U_t^\tau(t) = e^{ \A^\tau  t} [ \A^{\tau} U_0 -  F(0) ]  + \int_0^te ^{ \A^\tau  (t-\sigma) }  F_t(\sigma) ds  $.  
We shall apply the formula with $ F = \tau^{-1} [0,0, 2k ((u_t^{\tau})^2 + u^{\tau} u_{tt}^{\tau} )] $ where $u^{\tau}$ is the solution of the nonlinear problem.  
By Theorem 1.4   we have that nonlinear solutions $U^{\tau} \in C([0,T]; \mathbb{H}_2) $  for initial data also in $\mathbb{H}_2 $. For  such solutions in $\mathbb{H}_2 $ one has that $ F_t =  2k \tau^{-1} [0,0,(3u^\tau_tu^\tau_{tt} +u^\tau u_{ttt}^\tau)]  \in  L_{\infty}(\mathbb{H}_0) $ and $[A^{\tau} U_0 -  F(0) ]\in \mathbb{H}_0 $.  The latter 
can be deduced  directly from the regularity of solutions in $\mathbb{H}_2 $ and  Sobolev's embeddings in the dimensions of $\Omega$ less or equal to three.  This allows to consider the dynamics  in the variable $V^{\tau}=U^{\tau}_t \in \mathbb{H}_0$ which leads to a familiar MGT equation in $\mathbb{H}_0 $ 
 \ifdefined\xxxx
$F_t(U^\tau) =\tau^{-1}  [0,0, 2k(3u^\tau_tu^\tau_{tt} +u^\tau u_{ttt}^\tau) ]  \in \mathbb{H}_0 $,
 $$U^\tau(t)+ \tau^{-1} F(U^\tau), \ t>0,$$ 
An intermediate step to make the differentiation of \eqref{eqnl} not only formal is needed. For this, recall that if $X$ is a Banach space and $A: D(A) \subset X \to X$ generates a $C_0$-semigroup in $X$, then the \emph{linear}-problem $z_t = Az + h$ will have a solution $u \in C(0,T; X)$ provided $u(0) \in X$ and $F \in L^1(0,T; X).$ Therefore, the formal differentiation, $w_t = Aw + h_t$ ($w = z_t$) will have a solution $w \in C(0,T; H)$ provided $w(0) = u_t(0) = Au(0) + h(0) \in H$ and $h_t \in L^1(0,T; H).$

Back to our setting, recall that if $U_0 \in \mathbb{H}_2$, then the problem \eqref{eqnl} written as $$U^\tau_t(t) = \A^\tau U^\tau(t)+ \tau^{-1} F(U^\tau), \ t>0,$$ with $U^\tau(0)=U_0$ have a solution $U^\tau \in C(0,T; \mathbb{H}_2).$ Therefore, in view of the comments on the previous paragraph, differentiating the problem is justified with $U_t \in C(0,T; \mathbb{H}_1)$ provided $U_t(0) = \A^\tau U_0 + \tau^{-1}F(U_0) \in \mathbb{H}_1$ (which is the case since $U \in \mathbb{H}_2$ for each $t$) and $F_t(U^\tau) = 2k \tau^{-1} [0,0,(3u^\tau_tu^\tau_{tt} +u^\tau u_{ttt}^\tau)]  \in L^2(\Omega)$, which is again easy to see since $u^\tau, u_t^\tau \in \D{A} \hookrightarrow H^\infty(\Omega)$ and then multiplication by $L^2$-functions remains in $L^2(\Omega).$ Therefore $U_t \in C(0,T; \mathbb{H}_1)$. This justifies  \eqref{difve}.
\fi 
\begin{equation}\label{difve}
	\tau v^\tau_{ttt} + v^\tau_{tt} + c^2 Av^\tau + bAv^\tau_{t} = G'( u^\tau) = 2k(3u_t^\tau u_{tt}^\tau + u^\tau u^\tau_{ttt}) \in L_2(\Omega)
	\end{equation} 
	Repeating the calculations  leading to the proof of \eqref{id1}, but applied to \eqref{difve}   we obtain
	\begin{equation}\label{n1}
	\gamma^{\tau } \int_0^t ||v^{\tau}_{tt}||^2_2 d \sigma \leq |\int_0^t (G'(u^{\tau} ), m(\sigma) )  d\sigma |+ C [||A^{1/2} v^{\tau}(0) ||^2_2 +
	|| A^{1/2} v^{\tau}_{t}(0) ||^2_2 + \tau ||v^{\tau}_{tt}(0)  ||_2^2 ] 
	\end{equation}
	where $m(t) := v_{tt}^{\tau} + c^2 b^{-1} v_{t}^{\tau} $.
	\ifdefined\xxxxxx
	We then take the $L^2$-inner product of \eqref{difve} with $m(t) := u_{ttt}^{\tau} + c^2 b^{-1} u_{tt}^{\tau} = z^\tau_{tt}$ to get: \begin{align}
	\label{g2} (F'(u^\tau),z^\tau_{tt}) &= (\tau u^\tau_{tttt} + u^\tau_{ttt} + c^2 Au^\tau_t + bAu^\tau_{tt},u^\tau_{ttt}) + \dfrac{c^2}{b}\left[(\tau u^\tau_{tttt} + u^\tau_{ttt} + c^2 Au^\tau_t + bAu^\tau_{tt},u^\tau_{tt})\right] \nonumber \\ &= \dfrac{\tau}{2}\dfrac{d}{dt}\|u^\tau_{ttt}\|_2^2 + \|u^\tau_{ttt}\|_2^2 + c^2\dfrac{d}{dt}(Au^\tau_t,u^\tau_{tt})-c^2\|\p u^\tau_{tt}\|_2^2 + \dfrac{b}{2}\dfrac{d}{dt}\|\p u^\tau_{tt}\|_2^2 \nonumber \\ &+ \dfrac{c^2}{b}\left[\tau\dfrac{d}{dt}(u^\tau_{ttt},u^\tau_{tt}) - \tau\|u^\tau_{ttt}\|_2^2 +\dfrac{1}{2}\dfrac{d}{dt}\|u^\tau_{tt}\|_2^2 + \dfrac{c^2}{2}\dfrac{d}{dt}\|\p u^\tau_t\|_2^2 + b \|\p u^\tau_{tt}\|_2^2\right].
	\end{align} 
	\fi
	This leads to 
	\begin{equation}\label{g3}\int_0^t \gamma^\tau \|u^\tau_{ttt}(\sigma)\|_2^2 ds \leq C[  \tau \|u^\tau_{ttt}(0)\|_2^2  + \tau^{-1} \|U_0\|^2_{\mathbb{H}_2^{\tau}} +\int_0^t (G'(u^\tau(\sigma)),m(\sigma)) d\sigma].\end{equation}. It remains to estimate the nonlinear terms and $u_{ttt}^{\tau}(0) $. 
	For the nonlinear terms we apply the estimates in Theorem \ref{thm1} and Theorem \ref{thm2} . We thus obtain (for each $\varepsilon > 0$):\begin{align}
	\label{g4} (2k)^{-1}(G'(u^\tau) ), m) &= 3(u^\tau_t u^\tau_{tt}+u^\tau u^\tau_{ttt},u^\tau_{ttt}) + \dfrac{c^2}{b}(u^\tau_t u^\tau_{tt}+u^\tau u^\tau_{ttt},u^\tau_{tt}) \nonumber \\ & \leqslant C_1(\varepsilon)(\|u^\tau_t\|_{L^4}^2+\|u^\tau\|_{L^4}+\|u^\tau_t\|_2)\|u^\tau_{tt}\|_{L^4}^2 + (\varepsilon + \|u^\tau\|_{L^\infty})\|u^\tau_{ttt}\|_2^2 \nonumber \\ & \leqslant C_1(\varepsilon)(\|u^\tau_t\|_{H^1}^2+\|u^\tau\|_{L^4}^2+\|u^\tau_t\|_2)\|u^\tau_{tt}\|_{L^4}^2 + \left(\varepsilon + \|u^\tau\|_{H^2}^{1/2}\|u^\tau\|_{H^1}^{1/2}\right)\|u^\tau_{ttt}\|_2^2 \nonumber \\ &\leqslant C_1(\varepsilon) \He^\tau(0)[\E^\tau(0)]^{1/2}  + \left(\varepsilon + [\E^\tau(\sigma)]^{1/4}[E^\tau(\sigma)]^{1/4}\right) \|u^{\tau}_{ttt}(\sigma)\|_2^2,
	\end{align} for all $\sigma > 0.$ Returning  back in \eqref{g3} we have
	\begin{equation} \label{g5}
	\int_0^t \left(\gamma^\tau -\epsilon - [\E^\tau(\sigma)]^{1/2}[E(\sigma)]^{1/2}\right)\|u^\tau_{ttt}(\sigma)\|_2^2 d\sigma \leqslant C_1(\varepsilon)\mathfrak{E}^\tau(0)[\E^\tau(0)]^{1/2}  + C \tau^{-1} ||U_0||^2_{\mathbb{H}_2^{\tau}} +  C\tau \|u^{\tau}_{ttt}(0)\|_2^2,
	\end{equation}
	 Going back to the original equation, evaluating it at $t = 0$, 
	gives
	$$\|u^{\tau}_{ttt}(0)\|^2_2 \leq \tau^{-2} C [\|Au_0\|_2^2 + \|Au_1\|_2^2 + \|u_2\|_2^2 + \|u_0 u_2 + u_1^2\|^2]
	\leq  \tau^{-2} C [\|U_0\|^2_{\mathbb{H}_2}  + \|Au_0\|^2_2\|u_2\|^2_{2}  + \|A^{1/2} u_1\|_2^4]
	 $$
	\begin{align}
	 \|u^\tau_{ttt}(0)\|_2^2 \leqslant \tau^{-2} C(\|U_0\|_{\Hc}^2),
	\end{align} and going back to (\ref{n1}), recalling  positivity of $\gamma^\tau$, smallness of $\varepsilon$ and $\rho$ -- note that the smallness  of initial data  propagates  in time along the trajectories --  and boundedness of $\E^\tau$ $\He^{\tau} $ , the assertion follows.
\end{proof}

\begin{proof}[\bf Proof of Theorem \ref{t0}]
 We denote $2E_{u,0}(t) := \|u^0_t(t)\|_2^2 + \|\p u^0(0)\|_2^2$, $2E_{u,1}(t) = \|u^0_{tt}\|_2^2 + \|\p u^0_t(t)\|_2^2 + \|A u^0(t)\|_2^2$ and $2\E_u(t) := \|\p u^0_t(t)\|_2^2 + \|Au^0(t)\|_2^2.$ From  the assumption, then, we have, for each $M>0$ the existence of a sufficiently small $\rho_M$ such that $\E_u(0) \leqslant M$ and $E_{u,0}(0) \leqslant \rho_M.$ 	The following calculations   are patterned after \cite{WE} -- however the result proved below has  less regular initial data along with tighter control of the required smallness -- a fact which is needed for the ultimate result.  First, taking the $L^2$-inner product of \eqref{jmgt0} with $u^0_t(t) \in \D{A}$ we have \begin{align}
	k(u^0_t,(u^0_t)^2) &= \dfrac{d}{dt}\left[\dfrac{1}{2}\|u^0_t\|_2^2 + \dfrac{c^2}{2}\|\p u^0\|_2^2 + k(u^0,(u^0_t)^2)\right] + b\|\p u^0_t\|_2^2,
	\end{align} integration in time then implies \begin{align}\label{g7}
	\dfrac{1}{2}\|u^0_t\|_2^2 + \dfrac{c^2}{2}\|\p u^0\|_2^2 + b\int_0^t \|\p u^0_t(\sigma)\|_2^2d\sigma \leqslant C_1E_u(0) - k(u^0,(u^0_t)^2)\biggr\rvert_0^t - k\int_0^t(u^0_t(\sigma),(u^0_t(\sigma))^2)d\sigma,
	\end{align} and then the embeddings $H^1(\Omega) \hookrightarrow L^6(\Omega)$ and $H^{1/2}(\Omega) \hookrightarrow L^3(\Omega)$ allows us to rewrite \eqref{g7} as \begin{align}
(C_1-\|u^0\|_{L^\infty})\|u^0_t\|_2^2+ \dfrac{c^2}{2}\|\p u^0\|_2^2 &+ \int_0^t \left[b\|\p u^0_t(\sigma)- C_2 \|u^0_t(\sigma)\|_{H^{1/2}}^3\right]d\sigma \nonumber \\ &\leqslant C_1E_{u,0}(0) + k\|u^0(0)\|_{L^\infty}\|u^0_t(0)\|_2^2.
	\end{align}
	
	Next, we take the $L^2$-inner product of \eqref{jmgt0} with $Au^0(t) \in L^2(\Omega)$ we have \begin{align}
	(G(u^0),Au^0) = \dfrac{d}{dt}\left[(u^0_t,Au^0) +\dfrac{b}{2}\|Au^0\|_2^2\right]-\|\p u^0_t\|_2^2 + c^2\|Au^0\|_2^2,
	\end{align} and with a combination of H\"{o}lder's Inequality and Sobololev embeddings we get, after integration by parts, that for every $\varepsilon > 0$, \begin{align}
	C_1\|Au^0\|_2^2 + (c^2-\varepsilon)\int_0^t\|Au^0(\sigma)\|_2^2d\sigma &\leqslant C_2E_{u,0}(0) + \int_0^t \|\p u^0_t(\sigma)\|_2^2d\sigma \nonumber \\ & + C_3\|u^0_t\|_2^2 + C_4(\varepsilon)\int_0^t\left[ \|u^0\|_{L^\infty}^2\|u^0_{tt}(\sigma)\|_2^2 + \|u^0_t(\sigma)\|_{H^{3/4}}^4\right]d\sigma,
	\end{align} where we have used the embeddings $H^1(\Omega) \hookrightarrow L^4(\Omega)$ and $H^{3/4}(\Omega) \hookrightarrow L^2(\Omega)$. Finally, taking the $L^2$-inner product of \eqref{jmgt0} with $u^0_{tt}(t) \in L^2(\Omega)$ we have \begin{align}
	(G(u^0),u^0_{tt}) &= \dfrac{d}{dt} \left[c^2(u^0_t,Au^0)+\dfrac{b}{2}\|\p u^0_t\|_2^2\right]+\|u^0_{tt}\|_2^2-c^2\|\p u^0_t\|_2^2,
	\end{align} and then integrating from $0$ to $t \in [0,T]$ we have \begin{align}
	\int_0^t(1-2k\|Au^0(t)\|^{1/2} \|A^{1/2} u^0(t) \|^{1/2} )\|u^0_{tt}(\sigma)\|_2^2d\sigma &+ C_1\left[\|\p u^0_t(t)\|_2^2 - \|u^0_t(t)\|_{H^{1/2}}^3\right] \nonumber \\ &\leqslant C_2 E_{u,0}(0) + C_3\left[\int_0^t\|\p u^0(\sigma)\|_2^2d\sigma + \|\p u^0\|_2^2\right],
	\end{align} where we have used the embedding $H^{1/2}(\Omega) \hookrightarrow L^3(\Omega)$. Using the inequalities
	$||u||^2_{L^{\infty} } \leq C ||A u||^{1/2} ||A^{1/2} u||  \leq E_u^{1/2} \E_u^{1/2} $,  and $\|u^0_t\|_{H^{1/2}} \leqslant \|u^0_t\|_2^{1/2}\|\p u^0_t\|_2^{1/2}$ -- both dominated above by $E_{u,0}^{1/4}\E_u^{1/4}$ -- and adding the expressions generated by each multiplier, the estimate \eqref{lime} follows. The exponential decay follow from the given estimates  in a standard manner via Barrier's Method where we account for smallness of $E_u$ energy. \end{proof}

Now we are in position to prove  Theorem \ref{cc0}. \begin{proof}[\bf Proof of Theorem \ref{cc0}:] If $u^\tau$ is the solution of \eqref{eqnl} and $u^0$ is the solution \eqref{jmgt0}, then $x^\tau:= u^\tau - u^0$ solves the  following equation:
\begin{equation}\label{df1}
x^\tau_{tt}
+c^{2}Ax^\tau+bAx^\tau_{t}=-\tau u^\tau_{ttt}+2ku^\tau_{t}x^\tau_{t}+2ku^0_tx^\tau_{t}+2ku^\tau x^\tau_{tt}+2ku^0_{tt}x^\tau
\end{equation} with the  zero initial conditions. Let $L(t)$ and $R(t)$ denote the left hand side and right hand side of \eqref{df1}. We start by taking the $L^2$-inner product of \eqref{df1} with $Ax^\tau_t$ and $x^\tau_{tt}$, which -- for the left hand side -- give\begin{align}\label{l0}
\int_0^t(L(\sigma),Ax^\tau(\sigma))d\sigma &= \dfrac{1}{2}\|\p x^\tau_t\|_2^2 + \dfrac{c^2}{2}\|Ax^\tau\|_2^2 + b\int_0^t\|Ax^\tau_t(\sigma)\|_2^2d\sigma
\end{align} \begin{align}\label{l1}
\int_0^t(L(\sigma),x^\tau_{tt}(\sigma)) &= \dfrac{b}{2}\|\p x^\tau_t\|_2^2 +\int_0^t\|x^\tau_{tt}(\sigma)\|_2^2d\sigma + c^2\int_0^t(Ax^\tau(\sigma),x^\tau_{tt}(\sigma))d\sigma
\end{align} and, for every $\varepsilon >0,$ the right hand sides give \begin{align}
\label{l2} \int_0^t(R(\sigma),Ax^\tau_t(\sigma))d\sigma = \int_0^t\left(-\tau u^\tau_{ttt}+2ku^\tau_{t}x^\tau_{t}+2ku^0_tx^\tau_{t}+2ku^\tau x^\tau_{tt}+2ku^0_{tt}x^\tau,Ax^\tau_t\right)d\sigma \nonumber \\ \leqslant \tau\int_0^t(u^\tau_{ttt}(\sigma),Ax_t^\tau(\sigma))d\sigma + \varepsilon \int_0^t \|Ax^\tau_t(\sigma)\|_2^2d\sigma + C(\varepsilon)\int_0^t\left[\|u^\tau_t x^\tau_t\|_2^2 + \|u^\tau x^\tau_{tt}\|_2^2 + ||u^0_{tt}x^\tau\|_2^2 +||u_t^0 x_t||^2  \right]d\sigma
\end{align} and \begin{align}
\label{l3} \int_0^t(R(\sigma),x^\tau_{tt}(\sigma))d\sigma = \int_0^t\left(-\tau u^\tau_{ttt}+2ku^\tau_{t}x^\tau_{t}+2ku^0_tx^\tau_{t}+2ku^\tau x^\tau_{tt}+2ku^0_{tt}x^\tau,x^\tau_{tt}\right)d\sigma \nonumber \\ \leqslant \tau\int_0^t(u^\tau_{ttt}(\sigma),x^\tau_{tt}(\sigma))d\sigma + \varepsilon \int_0^t \|x^\tau_{tt}(\sigma)\|_2^2d\sigma + C(\varepsilon)\int_0^t\left[\|u^\tau_t x^\tau_t\|_2^2 + \|u^\tau x^\tau_{tt}\|_2^2 + ||u^0_{tt}x^\tau\|_2^2 +||u_t^0 x_t||^2 \right]d\sigma.
\end{align} Moreover, notice that by Lemma \ref{g1}  \begin{align}\label{rc1}
\int_{0}^{t}(\tau u^\tau_{ttt}(\sigma),Ax^\tau_{t}(\sigma) + x^{\tau}_{tt}(\sigma)  )d\sigma
&\leqslant \varepsilon\int_0^t [\|Ax^{\tau}_t(\sigma)\|_2^2 +\|x^{\tau}_{tt}(\sigma)\|_2^2] d\sigma + C_1(\varepsilon)\int_0^t \|\tau u^\tau_{ttt}(\sigma)\|_2^2d\sigma \nonumber \\ &\leqslant \varepsilon\int_0^t [\|Ax^{\tau}_t(\sigma)\|_2^2 +\|x^{\tau}_{tt}(\sigma)\|_2^2] d\sigma + \tau C_1(\varepsilon)(\|U_0\|_{\Hc}^2),
\end{align}  Then, adding  the corresponding sides, equating left and right  and taking $\varepsilon$ small we have \begin{align}
\label{l4} \E_{x^\tau}(t) + C_1\int_0^t[\|x^\tau_{tt}(\sigma)\|_2^2 &+ \|Ax^\tau_t(\sigma)\|_2^2]d\sigma \leqslant C_2 \int_0^t\|Ax^\tau(\sigma)\|_2^2d\sigma  +\tau C_3 (\|U_0\|_{\Hc}^2)  \nonumber \\ &+ C_4\int_0^t\left[\|u^\tau_t x^\tau_t\|_2^2 + \|u^\tau x^\tau_{tt}\|_2^2 +|| u^0_{tt}x^\tau\|_2^2 + ||u_t^0 x_t||^2 \right]d\sigma.
\end{align}
	We now estimate the four nonlinear terms. We have, \begin{align*}
	\int_0^t\|u^\tau(\sigma)x^\tau_{tt}(\sigma)\|_2^2d\sigma \leqslant  \int_0^t\|u^\tau\|_{L^\infty}^2 \|x^\tau_{tt}(\sigma)\|_2^2d\sigma &\leqslant C_1\int_0^t\|\p u^\tau\|_2\|Au^\tau\|_2 \|x^\tau_{tt}(\sigma)\|_2^2d\sigma \\ &\leqslant C_1[E^\tau(0)]^{1/2}[\E^\tau(0)]^{1/2}\int_0^t \|x^\tau_{tt}(\sigma)\|_2^2d\sigma.
	\end{align*} 
	where we have used a priori bounds for $u^{\tau}$ resulting from Theorem \ref{thm1}.  The later controls small  $\mathbb{H}_0$ norms of solutions in terms of  small initial data in that space (along with the  bounded (not small) $\mathbb{H}_1$-- norms.)
	Similarly,
	\begin{align*}
	\int_0^t \|u^0_{tt}(\sigma) x^\tau (\sigma)\|_2^2 d\sigma \leqslant \int_0^t \|x^\tau\|^2_{L^{\infty}}\|u^0_{tt}(\sigma)\|_2^2 d\sigma \leqslant C_1\int_0^t \|A^{1/2}x^\tau(\sigma)\|_2\|A x^\tau(\sigma)\|_2\|u_{tt}^0\|_2^2d\sigma  \nonumber \\
	\leqslant \varepsilon \sup\limits_{\sigma \in [0,T]}\|Ax^\tau(\sigma)\|_2^2\left(\int_0^t |u^0_{tt}(\sigma)|^2d\sigma\right)^2 + C_2(\varepsilon) \sup\limits_{\sigma\in [0,T]}\|A^{1/2}x^\tau(\sigma)\|_2^2,
	\end{align*}
	\begin{align*}
	\int_0^t \|u^\tau_t(\sigma)x^\tau_t(\sigma)\|_2^2d\sigma \leqslant \int_0^t \|x^\tau_t\|_{L^{\infty}}^2 \|u^\tau_t(\sigma)\|_2^2d\sigma \leqslant  C_1[E^\tau(0)]^{1/2}[\E^\tau(0)]^{1/2}\int_0^t \|Ax^\tau_t(\sigma)\|_2^2d\sigma,
	\end{align*} and finally \begin{align*}
	\int_0^t \|u_t^0(\sigma) x^\tau_t(\sigma)\|_2^2 ds \leqslant \int_0^t \|x^\tau_t\|_{L^{\infty} }^2 \|u_t^0(\sigma)\|_2^2d\sigma \leqslant C_1[E^\tau(0)]^{1/2}[\E^\tau(0)]^{1/2}\int_0^t \|Ax^\tau_t(\sigma)\|_2^2d\sigma.
	\end{align*}
	
	Now, taking $\varepsilon$ small and accounting, also, for smallness of $E^\tau(0)$ we rewrite \eqref{l4} as \begin{align}
	\label{l5} ||A x^{\tau}(t) ||^2_2 + ||A^{1/2} x^{\tau}_t (t)||^2_2 + C_1\int_0^t[\|x^\tau_{tt}(\sigma)\|_2^2 &+ \|Ax^\tau_t(\sigma)\|_2^2]d\sigma \leqslant \tau C_2 (\|U_0\|_{\Hc}^2) + C_3 \int_0^t\E_{x^\tau}(\sigma)d\sigma,
	\end{align} where we have used the fact that
	$$\|A^{1/2}x^\tau(t)\|_2^2 \leq  2 \int_0^t \|A^{1/2} x^\tau_t(\sigma)\|_2\|A^{1/2}x^\tau(\sigma)\|_2 d\sigma \leqslant C_1 \int_0^t [\|A^{1/2} x^\tau_t(\sigma)\|_2^2 + \|A^{1/2}x^\tau\|_2^2]d\sigma.$$ The final estimate then follows by Grownwall's Inequality, that is, \begin{align}\label{df3nnnn}
	 ||A x^{\tau}(t) ||^2_2 + ||A^{1/2} x^{\tau}_t (t)||^2_2 +C_1\int_0^t [\|x^{\tau}_{tt}(\sigma)|^2  + \|A x^{\tau}_t(\sigma)\|_2^2]d\sigma \leqslant \tau e^{\omega t}
	C_2 (\|U_0\|_{\Hc}^2) 
	\end{align}
%
	This completes the proof of  the rate of convergence. We reiterate that these results hold under the assumption that the energy is $\Ha$-small and $\Hc$-finite.
	 
	\begin{rmk}
		Note that the proof of convergence rates holds on every fixed time interval. In order to obtain the convergence on the entire $\mathbb{R}^+$ we need to appeal to the uniform decay rates of the solutions of limiting and the limit problems.
	\end{rmk}
	\ifdefined\xxxx

	Then from \eqref{rc1} with Sobolev's embeddings
	\begin{align*}
	\int_{0}^{t}(2ku^\tau_{t}(\sigma)x^\tau_{t}(\sigma),x^\tau_{t}(\sigma))d\sigma&\leqslant\int_{0}^{t}4k^2C_\varepsilon |u^\tau_{t}|^2_{H^1}|\p x^\tau_{t}(\sigma)|^2d\sigma+\int_{0}^{t}\varepsilon |x^\tau_{t}(\sigma)|^2d\sigma,\nonumber\\
	&\leqslant \int_{0}^{t}(4k^2C_\varepsilon|u^\tau_t|^2_{H1}+\varepsilon C^*)|\p x^\tau_t(\sigma)|^2d\sigma.
	\end{align*}
	\begin{align*}
	\int_{0}^{t}(2ku^0_{t}(\sigma)x^\tau_{t}(\sigma),x^\tau_{t}(\sigma))d\sigma&\leqslant\int_{0}^{t}4k^2C_\varepsilon |u^0_{t}|^2_{H^1}|\p x^\tau_{t}(\sigma)|^2d\sigma+\int_{0}^{t}\varepsilon |x^\tau_{t}(\sigma)|^2d\sigma,\nonumber\\
	&\leqslant \int_{0}^{t}(4k^2C_\varepsilon|u^0_t|^2_{H_1}+\varepsilon C^*)|\p x^\tau_t(\sigma)|^2d\sigma.
	\end{align*}
	\begin{align*}
	\int_{0}^{t}(2ku^\tau(\sigma) x^\tau_{tt}(\sigma),x^\tau_{t}(\sigma))d\sigma&=\int_{0}^{t}(2ku^\tau(\sigma),\dfrac{d}{dt}(x^\tau_{t})^2)d\sigma,\nonumber\\
	&=(2ku^\tau,(x^\tau_{t})^2)\bigg\vert_0^t-\int_{0}^{t}(2ku^\tau_{t}(\sigma)x^\tau_{t}(\sigma),x^\tau_{t}(\sigma))d\sigma,\nonumber\\
	&\leqslant4k^2\sup|u^\tau|_{L_\infty}| x^\tau_{t}(t)|^2+\int_{0}^{t}(2ku^\tau_{t}(\sigma)x^\tau_{t}(\sigma),x^\tau_{t}(\sigma))d\sigma,\nonumber\\
	&\leqslant4k^2\bar{C}| x^\tau_{t}(t)|^2+\int_{0}^{t}(2ku^\tau_{t}(\sigma)x^\tau_{t}(\sigma),x^\tau_{t}(\sigma))d\sigma,\nonumber\\
	&\leqslant4k^2\bar{C}| x^\tau_{t}(t)|^2+\int_{0}^{t}(4k^2C_\varepsilon|u^\tau_t|^2_{H1}+\varepsilon C^*)|\p x^\tau_t(\sigma)|^2d\sigma.
	\end{align*}
	where the constant $\bar{C}=\sup|u^\tau|_{L_\infty}$
	$$\int_{0}^{t}(2ku^0_{tt}(\sigma)x^\tau(\sigma),x^\tau_{t}(\sigma))d\sigma\leqslant\int_{0}^{t}4k^2C_\varepsilon|u^0_{tt}|^2_{L_2}|\p x^\tau(\sigma)|^2d\sigma+\int_{0}^{t}\varepsilon C^*|\p x^\tau_t(\sigma)|^2d\sigma.$$
	Then we get
	\begin{align}
	\text{RHS}&\leqslant\dfrac{\tau}{2}\big[2\tilde{M}+1+\hat{C}\big]\mathfrak{E}(0)+\dfrac{1}{4}|x^\tau_{t}(t)|^2+4k^2\bar{C}| x^\tau_{t}(t)|^2\nonumber\\
	&\quad+\int_{0}^{t}\bigg[4k^2C_\varepsilon(2|u^\tau_{t}|^2_{H^1}+|u^0_t|^2_{H^1}+4\varepsilon C^*)\bigg]|\p x^\tau_{t}(\sigma)|^2+\int_{0}^{t}4k^2C_\varepsilon|u^0_{tt}|^2_{L_2}|\p x^\tau(\sigma)|^2d\sigma,\nonumber\\
	&\leqslant\dfrac{\tau}{2}\big[2\tilde{M}+1+\hat{C}\big]\mathfrak{E}(0)+(\dfrac{1}{4}+4k^2\bar{C})|x^\tau_{t}(t)|^2+T4k^2C_\varepsilon|u^0_{tt}|^2_{L_2}|\p x^\tau|^2\nonumber\\
	&\quad+\int_{0}^{t}\bigg[4k^2C_\varepsilon(2|u^\tau_{t}|^2_{H^1}+|u^0_t|^2_{H^1}+4\varepsilon C^*)\bigg]|\p x^\tau_{t}(\sigma)|^2.
	\end{align}
	From \eqref{df3} we obtain
	\begin{align}\label{df4}
	(\dfrac{1}{4}-4k^2\bar{C})|x^\tau_{t}(t)|^2&+(\dfrac{c^2}{2}-T4k^2C_\varepsilon|u^0_{tt}|^2_{L_2})|\p x^\tau(t)|^2\nonumber\\
	&+\int_{0}^{t}\bigg[b-4k^2C_\varepsilon(2|u^\tau_{t}|^2_{H^1}+|u^0_t|^2_{H^1}+4\varepsilon C^*)\bigg]|\p x^\tau_{t}(\sigma)|^2d\sigma\nonumber\\
	&\leqslant\dfrac{\tau}{2}\big[2\tilde{M}+1+\hat{C}\big]\mathfrak{E}(0).
	\end{align}
	With smallness of initial data, $\mathfrak{E}(0)\leqslant\rho$ and \eqref{we1}, we have positive constants
	$$\hat{C}_1=\dfrac{c^2}{2}-T4k^2C_\varepsilon|u^0_{tt}|^2_{L_2}\leqslant\dfrac{c^2}{4}.$$
	$$\hat{C}_2=(\dfrac{1}{4}-4k^2\bar{C}).$$
	$$\hat{C}_3=b-4k^2C_\varepsilon(2|u^\tau_{t}|^2_{H^1}+|u^0_t|^2_{H^1}+4\varepsilon C^*)\leqslant\dfrac{b}{2}.$$
	Hence we arrive at 
	\begin{equation}\label{df7a}
	\hat{C}_1|\p x^\tau(t)|^2+\hat{C}_2| x^\tau_{t}(t)|^2\leqslant\dfrac{\tau}{2}\big[2\tilde{M}+1+\hat{C}\big]\mathfrak{E}(0)=\tau\tilde{K}\mathfrak{E}(0)=\tau C(r).
	\end{equation}
	Thus we obtain the rate $\tau\tilde{K}$ and this complete the proof of  theorem \eqref{thm3} part a).
	\fi
	\ifdefined\xxxxxxx
	Now we move our attention to proving  the strong convergence (Part (b) of Theorem \ref{thm3}). We start by recalling that given $\varepsilon > 0$  and  any $T > 0$, there exists $\tau \leqslant \tau_0(\varepsilon, T)$ we have \begin{equation}\label{eq1}
	\|P (U^\tau(t,U_0))- U^0(t,PU_0)\|_{ \D{A}\times \D{\p}}\leqslant\varepsilon, \ \ t \in [0,T]
	\end{equation} for any initial data $U_0\in \mathbb{H}_2$ which are  small in $\mathbb{H}_0$. Our goal is to prove that the above inequality holds for all initial data 
	$U_0. \in \mathbb{H}_1$  which remain  small in $\mathbb{H}_0$ and this is equivalent to showing that  for all $\epsilon > 0 $ there exists $\tau_0 > 0 $ so that 
	 for all $\tau \in (0, \tau_0 ) $  we have $\|(x^\tau,x^\tau_t)\|_{\D{A}\times \D{\p}} \leqslant \varepsilon$ for all $t >0$,  equivalently:
	\begin{equation}\label{eq3}\|A x^\tau(t)\|_2^2 + \|A^{1/2} x^\tau_t(t)\|_2^2 \leqslant  \varepsilon \ \ \mbox{for all } t. >0\end{equation}
	 Showing \eqref{eq3} will make explicit the need for developing existance, uniqueness and careful estimate analysis for three different topologies. The classic argument of \emph{of extending a given inequality by density} cannot be used when (same level) smallness of the data is required because approximation cannot be made by sequences of small elements. This issue has been dealt with, by requiring smallness in a lower topology, which is our setting.
	
	Let $U_0 \in \Hb$ and let $\{U_0^n\}_{n \in \mathbb{N}} \subset \Hc$ such that $U_0^n \to U_0$ in $\Hb$ with renormalized $\Ha$-norm sufficiently small. 
	
	Notice that, for each $n \in \mathbb{N}$ we, by virtue of \eqref{eq1} we have that : for each $\epsilon >0 $ there exists $\tau_0 $ such that for all 
	\begin{align}\label{eq4}
	\|PU^\tau(t,U_0) - U^0(t,PU_0)\|_{\D{A} \times \D{\p}} &\leqslant  \|PU^\tau(t,U_0)-PU^\tau(t,U_0^n)\|_{\D{A} \times \D{\p}} \nonumber + \dfrac{\varepsilon}{3}\\& +\|U^0(t,PU_0^n)-U^0(t,PU_0)\|_{\D{A} \times \D{\p}}. \end{align} We need to show that the two other components of the right hand side of \eqref{eq4} are also small (uniform in $\tau \in (0,\tau_0(t)]$ for the first term). For the first, letting $u_n^\tau$ denote the solution for \eqref{nll} corresponding to initial data $U_0^n$ and $u_0^\tau$ corresponding to initial data $U_0$ we have $w_n^\tau := u_n^\tau-u_0^\tau$ satisfying \eqref{nll} with $G(w_n^\tau)$ given by \begin{align}
	G(w_n^\tau) = {u_n^\tau}_{tt} w_n^\tau  + u_0^\tau({w_n^\tau}_{tt}) + {u_n^\tau}_t{w_n^\tau}_t + {u_0^\tau}_t{w_n^
	\tau}_t,
	\end{align} which needs to be tested against the multipliers ${w_n^\tau}_{tt}, {w_n^\tau}_t, {w_n^\tau}$ and $A{w_n^\tau}.$ Denoting by $\E^\tau_n$, $\E^\tau_0$ the $\Hb$-energy of solutions for \eqref{nll} for $U_0^n$ and $U_0$, respectively, then denoting by $E^\tau_n$ and $E^\tau_{00}$ the $\Ha$-energy of solutions for \eqref{nll} for $U_0^n$ and $U_0$, respectively and finally denoting by $E_{w}^\tau$ and $\E_w^\tau$ the $\Ha$ and $\Hb$-energy of the diference of solutions, respectively, we have, \begin{align*}
	(G({w_n^\tau}),{w_n^\tau}_{tt}) &= ({u_n^\tau}_{tt} w_n^\tau  + u_0^\tau({w_n^\tau}_{tt}) + {u_n^\tau}_t{w_n^\tau}_t + {u_0^\tau}_t{w_n^
		\tau}_t,{w_n^\tau}_{tt}) \\ &\leqslant \|{w_n^\tau}\|_{L^\infty}(C_1(\varepsilon)\|{u_n^\tau}_{tt}\|_2^2 + \varepsilon \|{w_n^\tau}_{tt}\|_2^2) + \|u_0^\tau\|_{L^\infty}\|{w_n^\tau}_{tt}\|_2^2 + \varepsilon \|{w_n^\tau}_{tt}\|_2^2 + C_2(\varepsilon)(\|{u_0^\tau}_t\|_2^2+\|{u_n^\tau}_t\|_2^2)\|{w_n^\tau}_t\|_2^2\\ &\leqslant C_1(\varepsilon)E_n^\tau(0)[\E_w^\tau(\sigma)]^{1/4}[E_w^\tau(\sigma)]^{1/4} + \varepsilon C_2[\E_w^\tau(\sigma)]^{5/4}[E_w^\tau(\sigma)]^{1/4} + C_3[\E_0^\tau(0)]^{1/4}[E_0^\tau(0)]^{1/4}\E_{w}^\tau(\sigma) \\ &+C_4(\varepsilon) (E_{00}^\tau(0)+E_n^\tau(0))\E_w^\tau(\sigma),
	\end{align*} for all $\sigma \in [0,T].$ The exact same reasoning provide similar estimates for the other multipliers. By using smallness of the sequence $U_0^n$ as well as smallness of $U_0$ in $\Ha$ we can show -- following the procedure for showing \eqref{m29} -- that the energy $\E_{w}^\tau$ is such that \begin{align}
	\E_w^\tau(t) + C_1\int_0^t \E^\tau_w(\sigma)d\sigma \leqslant C_2\E^\tau_w(0),
	\end{align} and this implies that the first term of the right hand side of \eqref{eq4} is dominated by $\varepsilon/3$ as long as $n$ is large enough. Very analogous -- therefore non included here -- strategy works also for proving that the last term is dominated by $\varepsilon/3$ for $n$ large. The strong convergence in finite time is then proved.
	
	For the infinity time case we appeal for the uniform (in $\tau$) exponential stability in the following manner: we let $T_\varepsilon$ denote a time big enough such that $$\int_{T_\varepsilon}^\infty \E_{x^\tau}(\sigma)d\sigma \leqslant C \int_{T_\varepsilon}^\infty e^{-\alpha \sigma}d\sigma = -\dfrac{C}{\alpha}e^{-\alpha\sigma}\biggr\rvert_{T_\varepsilon}^\infty = \dfrac{C}{\alpha}e^{-\alpha T_\varepsilon} \leqslant \dfrac{\varepsilon}{2},$$ which is possible since $2\E_{x^t}(t) \leqslant \|Au^\tau(t)\|_2^2 + \|Au^0(t)\|_2^2 + \|\p u^\tau_t(t)\|_2^2 + \|\p u^0_t(t)\|_2^2$ and therefore, uniform stablity is inherited. Finally, by taking $T = T_\varepsilon$ in our finite time proof, the result follows.
	\fi
	
{\bf Strong convergence.} Now we move our attention to proving  the strong convergence claimed in  Part (b) of Theorem \ref{thm3}. We start by recalling that given $\varepsilon > 0$  and  any $T > 0$, there exists $\tau \leqslant \tau_0(\varepsilon, T)$ we have \begin{equation}\label{eq1}
	\|P (U^\tau(t,U_0))- U^0(t,PU_0)\|_{ \D{A}\times \D{\p}}\leqslant \varepsilon/3 , \ \ t \in [0,T]
	\end{equation} for any initial data $U_0\in \mathbb{H}_2$ which are  small in $\mathbb{H}_0$. Our goal is to prove that the above inequality holds for all initial data 
	$U_0. \in \mathbb{H}_1$  which are small in $\mathbb{H}_0$ and this is equivalent to showing 
	\begin{equation}\label{eq3}\|A x^\tau(t)\|_2^2 + \|A^{1/2} x^\tau_t(t)\|_2^2 \leqslant  \varepsilon \ \ \mbox{for all } ~\tau\in \Lambda ,  t. >0\end{equation}
	\begin{rmk} Proving \eqref{eq3} as valid for the initial data in $\mathbb{H}_1 $   depends on  the density of $\mathbb{H}_2 \subset \mathbb{H}_1 $. 
	It is essential here that the smallness of initial data is required only in $\mathbb{H}_0 $ -since otherwise smallness of the approximation  in $\mathbb{H}_2 $ topology   can not be guaranteed without having  the smallness of the element  approximated. This is the part where the fact that the analysis requires only $\mathbb{H}_0$ smallness is critical. \end{rmk}
	
	
	The convergence stated  below follows essentially from the proofs of Theorem \ref{thm1} and Theorem \ref{t0}. For reader's convenience we will sketch the main  ingredients of the argument. 
	  \begin{lemma}\label{conv2}
		Let $U_0 \in \Hb$  with $||U_0||_{\mathbb{H}_0 } \leq \rho$ ; and $\{U_0\}_{n \in \mathbb{N}} \subset \Hc$ such that $\|U_0^n\|_{\Ha} \leqslant \ C\rho, $ with $\rho>0$ sufficiently small, and $U_0^n \to U_0$ in $\Hb$ as $n \to \infty.$ Then,  as $n \rightarrow \infty $ 
	we have \begin{itemize}
			\item[\emph{a)}] $PU^\tau(t,U_0^n) \to PU^\tau(t,U_0)$ in $P\Hb = \D{A} \times \D{\p}$, for every $\tau \in (0,\tau_0]$ and every $t \in [0,T], \ T > 0.$
			
			\item[\emph{b)}] $U^0(t,PU_0^n) \to U^0(t,PU_0)$ in $P\Hb = \D{A} \times \D{\p}$, for every $t \in [0,T], \ T > 0.$
		\end{itemize}
	\end{lemma} \begin{proof}
	Let $u_n^\tau(t) := U^\tau(t,U_0^n)$, $u_0^\tau(t):= U^\tau(t,U_0)$ and  $w_n^\tau := u_n^\tau-u_0^\tau.$ Notice that $w_n^\tau$ satisfies \eqref{nll} with $G(w_n^\tau)$ given by \begin{align}
	G(w_n^\tau) = {u_n^\tau}_{tt} w_n^\tau  + u_0^\tau({w_n^\tau}_{tt}) + ({u_n^\tau+u_0^t}_t){w_n^\tau}_t 
	\end{align} which needs to be tested against the multipliers ${w_n^\tau}_{tt}, {w_n^\tau}_t, {w_n^\tau}$ and $A{w_n^\tau}.$ First notice that for all $\varepsilon>0$ and $f \in L^2(\Omega)$ we have\begin{align*}
	(G(w_n^\tau),f) \leqslant C(\varepsilon)\left[\|{u_n^\tau}_{tt} w_n^\tau\|_2^2  + \|u_0^\tau({w_n^\tau}_{tt})\|_2^2 + \|({u_n^\tau+u_0^t}_t){w_n^\tau}_t\|_2^2 \right] + \varepsilon \|f\|_2^2 \\  \leqslant C(\varepsilon)(\|U_0^n\|_{\Ha}^2 + \|U_0\|_{\Ha}^2) \|({w_n^\tau},{w_n^\tau}_t,{w_n^\tau}_{tt})\|_{\Hb}^2 + \varepsilon\|f\|_2^2 \leqslant C(\varepsilon)\rho^2 \|({w_n^\tau},{w_n^\tau}_t,{w_n^\tau}_{tt})\|_{\Hb}^2 + \varepsilon\|f\|_2^2,
	\end{align*} therefore, all the nonlinear terms can be estimated above by $(C(\varepsilon)\rho^2 + \varepsilon)\|({w_n^\tau},{w_n^\tau}_t,{w_n^\tau}_{tt})\|_{\Hb}^2$ 
	[the terms $\epsilon ||f||^2_2 $ with $f $ equal  to one of $u_{tt}, Au, u_t $ are absorbed by the dissipation 
 due to positivity of  $\gamma^\tau, b, c > 0$ -- see steps 1 to 4 of the proof of Theorem \ref{thm1} on Section \ref{unlow}) which can absorb the energy terms since we have smallness of $\rho$ and we can choose $\varepsilon$ small. Therefore, $\E^\tau[w_n^\tau]$ (energy in $\Hb$ with respect to $w_n^\tau$) is such that \begin{align}
	\E^\tau[w_n^\tau](t) + C_1\int_0^t \E^\tau[w_n^\tau](\sigma)d\sigma \leqslant C_2\|U_0^n - U_0\|^2_{\Hb},
	\end{align} which implies part a). Part b) follows by  following  the proof of Theorem \ref{t0}  applied to equation \eqref{jmgt0}.
\end{proof} 
To finalize the proof of Theorem \ref{thm3} we need one more step based on the diagonal argument.

Let $U_0 \in \Hb$, $||U_0||_{\mathbb{H}_0} \leq \rho $  and let $\{U_0^n\}_{n \in \mathbb{N}} \subset \Hc$ such that $U_0^n \to U_0$ in $\Hb$ with  $\Ha$-norm sufficiently small $\leq C \rho $. Combining \eqref{eq1} and Lemma \ref{conv2} we obtain:\begin{align}\label{eq4}
	\|PU^\tau(t,U_0) - U^0(t,PU_0)\|_{\D{A} \times \D{\p}} \leqslant  \|PU^\tau(t,U_0)-PU^\tau(t,U_0^n)\|_{\D{A} \times \D{\p}} \nonumber \\  + \|PU^\tau(t,U_0^n)-PU^0(t,U_0^n)\|_{\D{A} \times \D{\p}} +\|U^0(t,PU_0^n)-U^0(t,PU_0)\|_{\D{A} \times \D{\p}}  \end{align} The strong convergence in finite time is then proved.

	For a given $\epsilon >0 $ we select $N$ (Lemma \ref{conv2}) such that  for all $t \in [0, T ] $
	\begin{eqnarray}\label{c1}
	||P U^{\tau} (t, U_0^N) - PU^{\tau} (t, U_0) ||_{ \D{A}\times \D{\p}} \leqslant \frac{\varepsilon}{3} , \ \tau\in \Lambda \nonumber \\
	||P U^{0} (t, U_0^N) - PU^{0} (t, U_0) ||_{ \D{A}\times \D{\p}} \leqslant \frac{\varepsilon}{3}.
	\end{eqnarray}
	 The continuity of $ P U^{\tau} (t, U_0) $ is uniform in $\tau \in \Lambda $, 
\ifdefined\xxxxxx
	Notice that the continuity of solutions with respect to $\Hb$-data (see \eqref{m29}) implies that, as $n \to \infty$, $U^\tau(t,U_0^n) \to U^\tau(t,U_0)$ in $\Hb$ for all $t \in [0,T]$ and $\tau \in (0,\tau_0(T)].$ Moreover, the error estimate \eqref{df3nnnn} implies that, as $n \to \infty$, $U^0(t,PU_0^n) \to U^0(t,PU_0)$ in $P\Hb = \D{A} \times \D{\p}$ for all $t \in [0,T].$ The rest of the argument -- for strong convergence in finite time -- follows in a standard manner. In fact, with $\varepsilon > 0$ small and fixed, $T>0$ arbritraty there exist $\tau_0(\varepsilon,T)$ and $n \in \mathbb{N}$ large such that
	\fi
	hence 
	 \begin{align}
	\|PU^\tau(t,U_0) - U^0(t,PU_0)\|_{\D{A} \times \D{\p}} \leqslant  \|PU^\tau(t,U_0)-PU^\tau(t,U_0^N)\|_{\D{A} \times \D{\p}} \nonumber \\ + \|PU^\tau(t,U_N) - U^0(t,PU_0^NN)\|_{\D{A} \times \D{\p}}\nonumber +\|U^0(t,PU_0^N)-U^0(t,PU_0)\|_{\D{A} \times \D{\p}} \leqslant \dfrac{\varepsilon}{3}+\dfrac{\varepsilon}{3}+\dfrac{\varepsilon}{3} = \varepsilon,
	\end{align} for all $t \in [0,T]$ and $\tau \in (0,\tau_0(\varepsilon,T)]$, where $\eqref{eq1}$ were used in the middle term. 
	
	For the infinity time case we appeal for the uniform (in $\tau$) exponential stability in the following manner: we let $T_\varepsilon$ denote a time big enough such that $$\int_{T_\varepsilon}^\infty \E_{x^\tau}(\sigma)d\sigma \leqslant C \int_{T_\varepsilon}^\infty e^{-\alpha \sigma}d\sigma ||U_0||_{\mathbb{H}_1}  \leqslant \dfrac{\varepsilon}{2},$$ which is possible since $2\E_{x^t}(t) \leqslant \|Au^\tau(t)\|_2^2 + \|Au^0(t)\|_2^2 + \|\p u^\tau_t(t)\|_2^2 + \|\p u^0_t(t)\|_2^2$ and therefore, uniform stability is inherited from the uniform stability asserted by Theorems \ref{thm1} and Theorem \ref{t0}.  Finally, by taking $T = T_\varepsilon$ in our finite time proof, the result of Theorem \ref{thm3}  follows.
	\end{proof}
	\ifdefined\xxxxx
	\section{Appendix}
	In what follows we use the following Lemma: \begin{lemma}\label{conv2}
		Let $U_0 \in \Hb$ and $\{U_0\}_{n \in \mathbb{N}} \subset \Hc$ such that $\|U_0^n\|_{\Ha} \leqslant \rho, $ with $\rho>0$ sufficiently small, and $U_0^n \to U_0$ in $\Hb$ as $n \to \infty.$ Then, if $P$ denotes projection on the first two coordinates, as $n \to \infty$ we have \begin{itemize}
			\item[\emph{a)}] $PU^\tau(t,U_0^n) \to PU^\tau(t,U_0)$ in $P\Hb = \D{A} \times \D{\p}$, for every $\tau \in (0,\tau_0]$ and every $t \in [0,T], \ T > 0.$
			
			\item[\emph{b)}] $U^0(t,PU_0^n) \to U^0(t,PU_0)$ in $P\Hb = \D{A} \times \D{\p}$, for every $t \in [0,T], \ T > 0.$
		\end{itemize}
	\end{lemma} \begin{proof}
	Let $u_n^\tau(t) := U^\tau(t,U_0^n)$, $u_0^\tau(t):= U^\tau(t,U_0)$ and  $w_n^\tau := u_n^\tau-u_0^\tau.$ Notice that $w_n^\tau$ satisfies \eqref{nll} with $G(w_n^\tau)$ given by \begin{align}
	G(w_n^\tau) = {u_n^\tau}_{tt} w_n^\tau  + u_0^\tau({w_n^\tau}_{tt}) + ({u_n^\tau+u_0^t}_t){w_n^\tau}_t 
	\end{align} which needs to be tested against the multipliers ${w_n^\tau}_{tt}, {w_n^\tau}_t, {w_n^\tau}$ and $A{w_n^\tau}.$ First notice that for all $\varepsilon>0$ and $f \in L^2(\Omega)$, inequality \eqref{m29} gives\begin{align*}
	(G(w_n^\tau,f) &\leqslant C(\varepsilon)\left[\|{u_n^\tau}_{tt} w_n^\tau\|_2^2  + \|u_0^\tau({w_n^\tau}_{tt})\|_2^2 + \|({u_n^\tau+u_0^t}_t){w_n^\tau}_t\|_2^2 \right] + \varepsilon \|f\|_2^2 \\ & \leqslant C(\varepsilon)(\|U_0^n\|_{\Ha}^2 + \|U_0\|_{\Ha}^2) \|({w_n^\tau},{w_n^\tau}_t,{w_n^\tau}_{tt})\|_{\Hb}^2 + \varepsilon\|f\|_2^2,\\ & \leqslant C(\varepsilon)\rho^2 \|({w_n^\tau},{w_n^\tau}_t,{w_n^\tau}_{tt})\|_{\Hb}^2 + \varepsilon\|f\|_2^2,
	\end{align*} therefore, all the nonlinear terms can be estimated above by $$(C(\varepsilon)\rho^2 + \varepsilon)\|({w_n^\tau},{w_n^\tau}_t,{w_n^\tau}_{tt})\|_{\Hb}^2$$ and the damping terms (remember $\gamma^\tau, b, c > 0$ and see steps 1 to 4 of the proof of Theorem \ref{thm1} on Section \ref{unlow}) can absorb the energy terms since we have smallness of $\rho$ and we can choose $\varepsilon$ small. Therefore, $\E^\tau[w_n^\tau]$ (energy in $\Hb$ with respect to $w_n^\tau$) is such that \begin{align}
	\E^\tau[w_n^\tau](t) + C_1\int_0^t \E^\tau[w_n^\tau](\sigma)d\sigma \leqslant C_2\|U_0^n - U_0\|_{\Hb},
	\end{align} which implies part a). Part b) follows by similar procedure done with respect to equation \eqref{jmgt0}.
\end{proof} Let $U_0 \in \Hb$ and let $\{U_0^n\}_{n \in \mathbb{N}} \subset \Hc$ such that $U_0^n \to U_0$ in $\Hb$ with renormalized $\Ha$-norm sufficiently small. Then, combining \eqref{eq1} and Lemma \ref{conv2} we have:\begin{align}\label{eq4}
	\|PU^\tau(t,U_0) - U^0(t,PU_0)\|_{\D{A} \times \D{\p}} &\leqslant  \|PU^\tau(t,U_0)-PU^\tau(t,U_0^n)\|_{\D{A} \times \D{\p}} \nonumber \\ & + \|PU^\tau(t,U_0^n)-PU^0(t,U_0^n)\|_{\D{A} \times \D{\p}}\\& +\|U^0(t,PU_0^n)-U^0(t,PU_0)\|_{\D{A} \times \D{\p}} \to 0, \end{align} as $n \to \infty.$ The strong convergence in finite time is then proved.
	
	For the infinity time case we appeal for the uniform (in $\tau$) exponential stability in the following manner: we let $T_\varepsilon$ denote a time big enough such that $$\int_{T_\varepsilon}^\infty \E_{x^\tau}(\sigma)d\sigma \leqslant C \int_{T_\varepsilon}^\infty e^{-\alpha \sigma}d\sigma = -\dfrac{C}{\alpha}e^{-\alpha\sigma}\biggr\rvert_{T_\varepsilon}^\infty = \dfrac{C}{\alpha}e^{-\alpha T_\varepsilon} \leqslant \dfrac{\varepsilon}{2},$$ which is possible since $2\E_{x^t}(t) \leqslant \|Au^\tau(t)\|_2^2 + \|Au^0(t)\|_2^2 + \|\p u^\tau_t(t)\|_2^2 + \|\p u^0_t(t)\|_2^2$ and therefore, uniform stablity is inherited. Finally, by taking $T = T_\varepsilon$ in our finite time proof, the result follows.
	\fi
	


\appendix

\section{Proof of Theorem \ref{wph2}} \label{app}

As a starting point, we are going to use (see \cite{bongarti,mgtp1,jmgt}), that given $f \in L^1(0,T; \D{\p})$ and $\alpha > 0$, the linear problem \begin{equation}
\label{alp} \tau u_{ttt} + \alpha u_{tt} + c^2A u + bA u_t = f,
\end{equation} is wellposed -- in the variable $U = (u,u_t,u_{tt})$ -- and exponentially stable (with rates independent of $\tau$ for $\tau$ small) for initial data in $\mathbb{H}_i$ ($i = 0,1,2$). This is to say that, denoting by $t \mapsto S(t)$ the semigroup generated by the evolution \eqref{alp}, there exist constants $\omega_i, M_i >0$, ($i = 0,1,2$) such that \begin{equation}
\label{expi} \|U(t)\|_{\mathbb{H}_i^\tau} = \|S(t)U_0\|_{\mathbb{H}_i^\tau} \leqslant M_ie^{-\omega_i t} \|U_0\|_{\mathbb{H}_i^\tau}.
\end{equation}

Define $X$ as the set \begin{equation*}
X = \left\{W = (w,w_t,w_{tt})^\top \in C(0,T; \mathbb{H}_2); \sup_{t \in [0,T]} \|W(t)\|_{\mathbb{H}_2^\tau} < \infty \ \mbox{and} \ \sup_{t \in [0,T]}\|W(t)\|_{\mathbb{H}_0^\tau} < \eta \right\}
\end{equation*}  ($\eta > 0$ will be taken to be sufficiently small later) and equip it with the norm \begin{equation*}\label{norm}\|W\|_{X}^2 := \sup_{t \in [0,T]} \|W(t)\|_{\mathbb{H}_2^\tau}^2.\end{equation*}

Recalling the interpolation inequalities \begin{equation}\|g\|_{L^\infty} \leqslant C\|g\|_{\D{\p}}^{1/2}\|g\|_{\D{A}}^{1/2}, \ g \in \D{A}\end{equation} and \begin{equation}\|g\|_{L^4} \leqslant C\|g\|_2^{1/4}\|g\|_{\D{\p}}^{3/4} \ \ g \in \D{\p}\end{equation} we observe that if $W = (w,w_t,w_{tt}) \in X$, it follows that ${w},{w}_t \in \D{A} \hookrightarrow L^\infty(\Omega)$ and $w_{tt} \in \D{\p} \hookrightarrow L^4(\Omega)$ ($n = 2,3,4$), for each $t \in [0,T]$, and therefore $f({w}) := 2k({w}_t^2 + {w}{w}_{tt}) \in C(0,T;\D{\p}).$ This means that for each $W \in X$, $f(w)$ qualifies to be the right hand side of \eqref{alp}, and therefore it makes well defined the application $\Upsilon$ that associates each $W \in X$ to the solution $(u,u_t,u_{tt})^\top = U := \Upsilon(W) \in C(0,T; \mathbb{H}_2)$ for \eqref{alp} with initial condition $U_0 = (u(0),u_t(0),u_{tt}(0)) ^\top \in \mathbb{H}_2$. Moreover, the solution $U$ is represented by the variation of parameters formula, i.e., for each $t \in [0,T],$\begin{equation}
\label{wpf} U(t) = \Upsilon(W)(t) = S(t)U_0 + \int_0^t S(t-\sigma)\underbrace{(0,0,\tau^{-1}f(w(t)))^\top}_{:=F_\tau(W)(t)} d\sigma.
\end{equation} 

In addition, $\Upsilon$ maps $X$ into itself. In fact, for each $t \in [0,T]$, uniform (in $\tau$) exponential stability implies that\begin{align}
\|\Upsilon(W)(t)\|_{\mathbb{H}_2^\tau} & \leqslant \|S(t)U_0\|_{\mathbb{H}_2^\tau} + \left\|\int_0^tS(t-\sigma)F_\tau(W)(\sigma)d\sigma\right\|_{\mathbb{H}_2^\tau} \nonumber \\ &\leqslant M_2\|U_0\|_{\mathbb{H}_2^\tau} + \int_0^t M_2e^{-\omega_2 (t-\sigma)} \|F_\tau(W)(\sigma)\|_{\mathbb{H}_2^\tau}d\sigma \nonumber\\  &\leqslant M_2\left(\|U_0\|_{\mathbb{H}_2^\tau} + \dfrac{C_{\omega}}{\tau}\sup_{t \in [0,T]} \|f(w)(t)\|_{\D{\p}}\right). \label{mm}
\end{align} and again for each $t \in [0,T]$ -- mostly omitted on the computations below --  we estimate \begin{align}(2k)^{-1}\|f({w})\|_{\D{\p}} &\sim \|\nabla({w}_t^2 + {w} {w}_{tt})\|_2 \nonumber \\ &= \|2{w}_t \nabla {w}_t + {w} \nabla {w}_{tt} + {w}_{tt} \nabla {w}\|_2 \nonumber \\ &\leqslant 2\|{w}_t\|_{L^\infty} \|\nabla {w}_t\|_2 + \|{w}\|_{L^\infty}\|\nabla {w}_{tt}\|_2 + \|{w}_{tt}\|_{L^4}\|\nabla {w}\|_{L^4} \nonumber \\ &\leqslant C \left[\|{w}_t\|_{\D{\p}}^{1/2}\|{w}_t\|_{\D{A}}^{1/2}\|\nabla {w}_t\|_2 + \|{w}\|_{\D{\p}}^{1/2}\|{w}\|_{\D{A}}^{1/2}\|\nabla {w}_{tt}\|_2\right] \nonumber \\ &+ C\left[\|{w}_{tt}\|_2^{1/4}\|{w}_{tt}\|_{\D{\p}}^{3/4}\|\nabla {w}\|_2^{1/4}\|\nabla {w}\|_{\D{\p}}^{3/4}\right] \nonumber \\ &\leqslant C\|W(t)\|_{\mathbb{H}_0^\tau}^{1/2}\|W(t)\|_{\mathbb{H}_2^\tau}^{3/2} \leqslant \left[\sup_{t \in [0,T]}\|W(t)\|_{\mathbb{H}_2^\tau}\right]^{3/2}\eta^{1/2},\end{align} and then, back in \eqref{mm} we conclude that \begin{equation}\label{invar}\sup_{t \in [0,T]}\|\Upsilon(W)(t)\|_{\mathbb{H}_2^\tau} \leqslant M_2\left(\|U_0\|_{\mathbb{H}_2^\tau} + \dfrac{2k C_{\omega}}{\tau}\left[\sup_{t \in [0,T]}\|W(t)\|_{\mathbb{H}_2^\tau}\right]^{3/2}\eta^{1/2}\right) < \infty.\end{equation}

Similarly for $\mathbb{H}_0^\tau$, for each $t \in [0,T]$ we have have\begin{equation}\label{invar2}
\|\Upsilon(W)(t)\|_{\mathbb{H}_0^\tau} \leqslant M_1\left(\|U_0\|_{\mathbb{H}_0^\tau} + \dfrac{C_{\omega}}{\tau}\sup_{t \in [0,T]}\|f(w)(t)\|_2\right).
\end{equation} Moreover, for each $t \in [0,T]$ -- again mostly ommited -- it holds that \begin{align*}
(2k)^{-1}\|f(w)(t)\|_2 &= \|w_t^2 + uu_{tt}\|_2 \\ &\leqslant \|w_t\|_{L^4}^2 + \|u\|_{L^\infty}\|u_{tt}\|_2 \\ &\leqslant C\left[\|\p w_t\|_2^2 + \|\p u\|_2^{1/2}\|Au\|_2^{1/2}\|u_{tt}\|_2\right] \\ &\leqslant C\left\{\left[\sup_{t \in [0,T]}\|W(t)\|_{\mathbb{H}_0^\tau}\right]^2 + \left[\sup_{t \in [0,T]}\|W(t)\|_{\mathbb{H}_0^\tau}\right]^{3/2}\left[\sup_{t \in [0,T]}\|W(t)\|_{\mathbb{H}_2^\tau}\right]^{1/2}\right\} \\ &\leqslant C\left\{\eta^2 + \eta^{3/2}\left[\sup_{t \in [0,T]}\|W(t)\|_{\mathbb{H}_2^\tau}\right]^{1/2}\right\}, 
\end{align*} then, taking $\eta = \eta(\tau)$ small enough so that $$\dfrac{2kC}{\tau}\left[\eta^2 + \eta^{3/2}\left(\sup_{t \in [0,T]}\|W(t)\|_{\mathbb{H}_2^\tau}\right)^{1/2}\right] < \dfrac{\eta}{2M1}$$ and, as a consequence, $\rho = \rho(\tau)$ small enough such that $\rho < \dfrac{\eta}{2M1}$, we can return to \eqref{invar2} to conclude that if $\|U_0\|_{\mathbb{H}_0^\tau} < \rho$ we have \begin{equation}\label{invar22}\sup_{t \in [0,T]}\|\Upsilon(W)(t)\|_{\mathbb{H}_0^\tau} \leqslant M_1\left\{\rho + \dfrac{2kC_{\omega}}{\tau}\left[\eta^2 + \eta^{3/2}\left(\sup_{t \in [0,T]}\|W(t)\|_{\mathbb{H}_2^\tau}\right)^{1/2}\right]\right\} < \eta,\end{equation} which completes the proof of invariance.

Next we prove that $\Upsilon$ is a contraction if $\eta$ is sufficiently small. For this, let $W_1 = (w_1,{w_1}_t,{w_1}_{tt})^\top, W_2 = (w_2, {w_2}_{t},{w_2}_{tt})^\top \in X$ and notice that,
\begin{align}
\label{sme2} \|\Upsilon(W_1) - \Upsilon(W_2)\|_{X}  =\sup_{t \in [0,T]}\left\|\int_0^t  S(t-\sigma) \left[F_\tau(W_1)(\sigma)-F_\tau(W_2)(\sigma)\right]d\sigma\right\|_{\mathbb{H}_2^\tau}\nonumber \\
  \leqslant \dfrac{C_{\omega}}{\tau} \sup_{t \in [0,T]} \left\|f(w_1)(t)-f(w_2)(t)\right\|_{\D{\p}}. \end{align}

 Next, observe that for each $t \in [0,T]$ -- mostly omitted on the computations below -- we have \begin{align}(2k)^{-1}\|f(w_1)(t) -f(w_2)(t)\|_{\D{\p}} &=\|({w_1}_t+{w_2}_t)({w_1}_t - {w_2}_{t}) + ({w_1}-{w_2}){w_1}_{tt} + ({w_1}_{tt}-{w_2}_{tt}){w_2}\|_{\D{\p}} \nonumber \\ &\leqslant Q_1(t) + Q_2(t) + Q_3(t),
\end{align} 

where $$Q_1(t) = \left\|({w_1}_t+{w_2}_t)({w_1}_t-{w_2}_t)\right\|_{\D{\p}}, \ \ Q_2(t) = \|({w_1}-{w_2}){w_1}_{tt}\|_{\D{\p}}, \ \ Q_3(t) = \|({w_1}_{tt}-{w_2}_{tt}){w_2}\|_{\D{\p}}$$ and then we estimate these three quantities (for each $t \in [0,T]$):\begin{align*}
Q_1(t) &= \left\|({w_1}_t+{w_2}_t)({w_1}_t-{w_2}_t)\right\|_{\D{\p}}  \nonumber \\ & \sim \left\|\nabla\left[({w_1}_t+{w_2}_t)({w_1}_t-{w_2}_t)\right]\right\|_{2} \nonumber \\ & = \left\|({w_1}_t+{w_2}_t)\nabla({w_1}_t-{w_2}_t)\right\|_{2}+\left\|\nabla\left[({w_1}_t+{w_1}_t)\right]({w_1}_t-{w_2}_t)\right\|_{2}\nonumber \\ &\leqslant \|{w_1}_t +{w_2}_t\|_{L^\infty}\|{w_1}_t - {w_2}_t\|_{\D{\p}} + \|\nabla ({w_1}_t + {w_2}_t)\|_{L^4}\|{w_1}_t - {w_2}_t\|_{L^4} \nonumber \\ &\leqslant C\|{w_1}_t +{w_2}_t\|_{\D{\p}}^{1/2}\|{w_1}_t +{w_2}_t\|_{\D{A}}^{1/2}\|{w_1}_t - {w_2}_t\|_{\D{\p}} \nonumber \\&+ C\|\nabla ({w_1}_t + {w_2}_t)\|_{2}^{1/4}\|\nabla ({w_1}_t + {w_2}_t)\|_{\D{\p}}^{3/4}\|{w_1}_t - {w_2}_t\|_{2}^{1/4}\|{w_1}_t - {w_2}_t\|_{\D{\p}}^{3/4}
\nonumber \\ &\leqslant C\left[\|(W_1+W_2)(t)\|_{\mathbb{H}_0^\tau}\right]^{1/2}\left[\|(W_1+W_2)(t)\|_{\mathbb{H}_2^\tau}\right]^{1/2}\|({W_1}-{W_2})(t)\|_{\mathbb{H}_2^\tau} \nonumber \\ &+C\left[\|(W_1+W_2)(t)\|_{\mathbb{H}_0^\tau}\right]^{1/4}\left[\|(W_1+W_2)(t)\|_{\mathbb{H}_2^\tau}\right]^{3/4} \|({W_1}-{W_2})(t)\|_{\mathbb{H}_2^\tau} \nonumber \\&\leqslant C\left[\sup_{t \in [0,T]}\|(W_1+W_2)(t)\|_{\mathbb{H}_0^\tau}\right]^{1/2}\left[\sup_{t \in [0,T]}\|(W_1+W_2)(t)\|_{\mathbb{H}_2^\tau}\right]^{1/2}\sup_{t \in [0,T]}\|({W_1}-{W_2})(t)\|_{\mathbb{H}_2^\tau} \nonumber \\ &+C\left[\sup_{t \in [0,T]}\|(W_1+W_2)(t)\|_{\mathbb{H}_0^\tau}\right]^{1/4}\left[\sup_{t \in [0,T]}\|(W_1+W_2)(t)\|_{\mathbb{H}_2^\tau}\right]^{3/4} \sup_{t \in [0,T]}\|({W_1}-{W_2})(t)\|_{\mathbb{H}_2^\tau} \nonumber \\&\leqslant C\left\{\eta^{1/2}\left[\sup_{t \in [0,T]}\|(W_1+W_2)(t)\|_{\mathbb{H}_2^\tau}\right]^{1/2} + \eta^{1/4}\left[\sup_{t \in [0,T]}\|(W_1+W_2)(t)\|_{\mathbb{H}_2^\tau}\right]^{3/4}\right\}\|W_1-W_2\|_X \nonumber \\ &\leqslant 
C\eta^{1/4}\|W_1 - W_2\|_X,
\end{align*}where $C$ is a constant that does not depend on time.
\begin{align*}
Q_2(t) &=\|({w_1}-{w_2}){w_1}_{tt}\|_{\D{\p}}  \nonumber \\&\sim \|{w_1}_{tt}\nabla({w_1}-{w_2}) + ({w_1} - {w_2})\nabla {w_1}_{tt}\|_2 \\ & \leqslant C\|{w_1}_{tt}\|_2^{1/4}\|{w_1}_{tt}\|_{\D{\p}}^{3/4}\|\nabla({w_1}-{w_2})\|_{2}^{1/4}\|\nabla({w_1}-{w_2})\|_{\D{\p}}^{3/4} \\ &+ C\|{w_1}-{w_2}\|_{\D{\p}}^{1/2}\|{w_1}-{w_2}\|_{\D{A}}^{1/2}\|\nabla {w_1}_{tt}\|_2 \nonumber \\ &\leqslant \dfrac{C}{\tau}[\|W_1(t)\|_{\Ha^\tau}]^{1/4}[\|W_1(t)\|_{\Hc^\tau}]^{3/4}\|(W_1-W_2)(t)\|_{\mathbb{H}_2^\tau}+\dfrac{C}{\tau}[\|(W_1-W_2)(t)\|_{\Ha^\tau}]^{1/2}\|(W_1-W_2)(t)\|_{\mathbb{H}_2^\tau} \nonumber \\& + \dfrac{C}{\tau} [\|(W_1-W_2)(t)\|_{\Ha^\tau}]^{1/2}[\|(W_1-W_2)(t)\|_{\Hc^\tau}] \|(W_1-W_2)(t)\|_{\mathbb{H}_2^\tau}\nonumber \\ &\leqslant \dfrac{C}{\tau}\left[\sup_{t \in [0,T]}\|W_1(t)\|_{\Ha^\tau}\right]^{1/4}\left[\sup_{t \in [0,T]}\|W_1(t)\|_{\Hc^\tau}\right]^{3/4}\sup_{t \in [0,T]}\|(W_1-W_2)(t)\|_{\mathbb{H}_2^\tau}\nonumber \\ &+ \dfrac{C}{\tau}\left[\sup_{t \in [0,T]}\|(W_1-W_2)(t)\|_{\Ha^\tau}\right]^{1/2}\sup_{t \in [0,T]}\|(W_1-W_2)(t)\|_{\mathbb{H}_2^\tau} \nonumber \\& + \dfrac{C}{\tau}\left[\sup_{t \in [0,T]}\|(W_1-W_2)(t)\|_{\Ha^\tau}\right]^{1/2}\left[\sup_{t \in [0,T]}\|(W_1-W_2)(t)\|_{\Hc^\tau}\right] \sup_{t \in [0,T]}\|(W_1-W_2)(t)\|_{\mathbb{H}_2^\tau}
\\ &\leqslant \dfrac{C}{\tau}\left\{\eta^{1/4}\left[\sup_{t \in [0,T]}\|W_1(t)\|_{\Hc^\tau}\right]^{3/4}+ \eta^{1/2} + \eta^{1/2}\left[\sup_{t \in [0,T]}\|(W_1-W_2)(t)\|_{\Hc^\tau}\right]\right\}\|W_1-W_2\|_X \nonumber \\&  
\leqslant \dfrac{C}{\tau}\eta^{1/4}\|W_1 - W_2\|_X,
\end{align*} where $C$ is a constant that does not depend on time.\begin{align*}
Q_3(t) &= \|({w_1}_{tt}-{w_2}_{tt}){w_2}\|_{\D{\p}} \nonumber \\ &\sim \|{w_2}\nabla({w_1}_{tt}-{w_2}_{tt}) + ({w_1}_{tt} - {w_2}_{tt})\nabla {w_2}\|_2 \\ & \leqslant C\|{w_1}_{tt}-{w_2}_{tt}\|_2^{1/4}\|{w_1}_{tt}-{w_2}_{tt}\|_{\D{\p}}^{3/4}\|\nabla{w_2}\|_{2}^{1/4}\|\nabla{w_2}\|_{\D{\p}}^{3/4} \\ &+ C\|{w_2}\|_{\D{\p}}^{1/2}\|{w_2}\|_{\D{A}}^{1/2}\|\nabla({w_1}_{tt}-{w_2}_{tt})\|_2 \\ &\leqslant \dfrac{C}{\tau}\left\{[\|W_1\|_{\Ha^\tau}]^{1/4}[\|W_1(t)\|_{\Hc^\tau}]^{3/4}+ [\|(W_1-W_2)(t)\|_{\Ha^\tau}]^{1/2}[\|(W_1+W_2)(t)\|_{\Hc^\tau}]^{1/2}\right\} \|({W_1}-{W_2})(t)\|_{\Hc^\tau} \\ &\leqslant \dfrac{C}{\tau}\left\{\eta^{1/4}\left[\sup_{t \in [0,T]}\|W_1(t)\|_{\Hc^\tau}\right]^{3/4} + \eta^{1/2}\left[\sup_{t \in [0,T]}\|(W_1+W_2)(t)\|_{\Hc^\tau}\right]^{1/2}\right\}\leqslant \dfrac{C}{\tau}\eta^{1/4}\|W_1 - W_2\|_X,
\end{align*} where $C$ is a constant that does not depend on time.

Therefore, back in \eqref{sme2} we have \begin{align}
\label{sme3} \|\Upsilon(W_1) - \Upsilon(W_2)\|_{X} & \leqslant \dfrac{C_{\omega}}{\tau} \sup_{t \in [0,T]} \left\|f(w_1)(t)-f(w_2)(t)\right\|_{\D{\p}} \nonumber \\&\leqslant \dfrac{2kC_{\omega}}{\tau}\sup_{t \in [0,T]}\left[Q_1(t) + Q_2(t) + Q_3(t)\right] \leqslant \dfrac{2kTC}{\tau^2}\eta^{1/4}\|W_1-W_2\|_X.\end{align} which means that $\Upsilon$ is a contraction as long as we take a -- possibly smaller -- $\eta=\eta(\tau)$ such that $\eta < \left(\dfrac{\tau^2}{2kC_{\omega}}\right)^{16}$. This completes the proof of local wellposedness.

\bibliographystyle{acm} 
\bibliography{ref22.bib}

\begin{thebibliography}{10}

\bibitem{kaltev2}


\bibitem{bongarti}
{\sc Bongarti, M., Charoenphon, S., and Lasiecka, I.}
\newblock Singular thermal relaxation limit for the {Moore}-{Gibson}-{Thompson}
  equation arising in propagation of acoustic waves.
\newblock {\em Semigroups of Operators: Theory and Applications SOTA-2018\/}
  (2019), 147--182.
\newblock Publisher: Springer.

\bibitem{bucci2020}
{\sc Bucci, F., and Eller, M.}
\newblock The {Cauchy}–{Dirichlet} problem for the
  {Moore}–{Gibson}–{Thompson} equation.
\newblock {\em arXiv preprint arXiv:2004.11167\/} (2020).

\bibitem{bucci2017}
{\sc Bucci, F., and Pandolfi, L.}
\newblock On the regularity of solutions to the {Moore}–{Gibson}-{Thompson}
  equation: a perspective via wave equations with memory.
\newblock {\em Journal of Evolution Equations 20}, 3 (2019), 1--31.
\newblock Publisher: Springer.

\bibitem{ap5}
{\sc Cattaneo, C.}
\newblock Sulla conduzione del calore.
\newblock {\em Atti Sem. Mat. Fis. Univ. Modena 3\/} (1948), 83--101.

\bibitem{ap4}
{\sc Cattaneo, C.}
\newblock A form of heat–conduction equations which eliminates the paradox of
  instantaneous propagation.
\newblock {\em Comptes Rendus 247\/} (1958), 431.

\bibitem{HCP}
{\sc Christov, C., and Jordan, P.}
\newblock Heat conduction paradox involving second–sound propagation in
  moving media.
\newblock {\em Physical review letters 94}, 15 (2005), 154301.
\newblock Publisher: APS.

\bibitem{crig}
{\sc Crighton, D.~G.}
\newblock Model equations of nonlinear acoustics.
\newblock {\em Annual Review of Fluid Mechanics 11}, 1 (1979), 11--33.
\newblock Publisher: Annual Reviews 4139 El Camino Way, PO Box 10139, Palo
  Alto, CA 94303-0139, USA.

\bibitem{mem3}
{\sc Dell'Oro, F., Lasiecka, I., and Pata, V.}
\newblock The {Moore}–{Gibson}–{Thompson} equation with memory in the
  critical case.
\newblock {\em Journal of Differential Equations 261}, 7 (2016), 4188--4222.
\newblock Publisher: Elsevier.

\bibitem{pata}
{\sc Dell'Oro, F., and Pata, V.}
\newblock On the {Moore}–{Gibson}–{Thompson} equation and its relation to
  linear viscoelasticity.
\newblock {\em Applied Mathematics and Optimization 261}, 7 (2016), 4188--4222.
\newblock Publisher: Elsevier.

\bibitem{hieber}
{\sc Denk, R., Hieber, M., and Pruss, J.}
\newblock R–boundedness, {Fourrier} multipliers and problems of elliptic and
  parabolic type.
\newblock {\em Memoires of American Mathematical Society 788\/} (2003).
\newblock Publisher: AMS.

\bibitem{ap2}
{\sc Ekoue, F., d’Halloy, A.~F., Gigon, D., Plantamp, G., and Zajdman, E.}
\newblock Maxwell–{Cattaneo} regularization of heat equation.
\newblock {\em World Academy of Science, Engineering and Technology 7\/}
  (2013), 05--23.

\bibitem{fatto}
{\sc Fattorini, H.~O.}
\newblock {\em The {Cauchy} {Problem}}.
\newblock Addison Wesley, 1983.

\bibitem{ham}
{\sc Hamilton, M.~F., Blackstock, D.~T., and {others}}.
\newblock {\em Nonlinear acoustics}.
\newblock Academic Press, 1997.

\bibitem{Jordan2}
{\sc Jordan, P.~M.}
\newblock Nonlinear acoustic phenomena in viscous thermally relaxing fluids:
  {Shock} bifurcation and the emergence of diffusive solitons.
\newblock {\em The Journal of the Acoustical Society of America 124}, 4 (2008),
  2491--2491.
\newblock Publisher: ASA.

\bibitem{Jordan0}
{\sc Jordan, P.~M.}
\newblock Second-sound phenomena in inviscid, thermally relaxing gases.
\newblock {\em Discrete \& Continuous Dynamical Systems-B 19}, 7 (2014), 2189.
\newblock Publisher: American Institute of Mathematical Sciences.

\bibitem{kalt}
{\sc Kaltenbacher, B.}
\newblock Mathematics of nonlinear acoustics.
\newblock {\em Evolution Equations and Control Theory 4}, 4 (2015), 447--491.

\bibitem{WE}
{\sc Kaltenbacher, B., and Lasiecka, I.}
\newblock Global existence and exponential decay rates for the {Westervelt}'s
  equation.
\newblock {\em Discrete and Continuous Dynamical Systems–Series S 2}, 3
  (2009), 503--525.

\bibitem{mgtp1}
{\sc Kaltenbacher, B., Lasiecka, I., and Marchand, R.}
\newblock Wellposedness and exponential decay rates for the
  {Moore}–{Gibson}–{Thompson} equation arising in high intensity
  ultrasound.
\newblock {\em Control and Cybernetics 40\/} (2011), 971--988.

\bibitem{jmgt}
{\sc Kaltenbacher, B., Lasiecka, I., and Pospieszalska, M.~K.}
\newblock Well-posedness and exponential decay of the energy in the nonlinear
  {Jordan}–{Moore}–{Gibson}–{Thompson} equation arising in high intensity
  ultrasound.
\newblock {\em Mathematical Models and Methods in Applied Sciences 22}, 11
  (2012), 1250035.
\newblock Publisher: World Scientific.

\bibitem{kaltev1}
{\sc Kaltenbacher, B., and Nikolić, V.}
\newblock Vanishing relaxation time limit of the
  {Jordan}–{Moore}–{Gibson}–{Thompson} wave equation with {Neumann} and
  absorbing boundary conditions.
\newblock {\em Pure and Applied Functional Analysis 5\/} (2020), 1--26.

\bibitem{kato}
{\sc Kato, T.}
\newblock {\em Perturbation {Theory} for {Linear} {Operators}}.
\newblock Springer-Verlag Berlin Heidelberg, 1976.

\bibitem{ong}
{\sc Lasiecka, I., and Ong, J.}
\newblock Global solvability and uniform decays of solutions to quasilnear
  hyperbolic equations with nonlinear boundary conditions.
\newblock {\em Communications on PDE 24}, 11-12 (1999), 2069--2107.
\newblock Publisher: Francis and Taylor.

\bibitem{las-tat}
{\sc Lasiecka, I., Tataru, D., and {others}}.
\newblock Uniform boundary stabilization of semilinear wave equations with
  nonlinear boundary damping.
\newblock {\em Differential and integral Equations 6}, 3 (1993), 507--533.
\newblock Publisher: Khayyam Publishing, Inc.

\bibitem{lunardi}
{\sc Lunardi, A.}
\newblock {\em Analytic {Semigroups} and {Optimal} regularity in parabolic
  problems}.
\newblock Birkhäuser, 1995.

\bibitem{TrigMGT}
{\sc Marchand, R., McDevitt, T., and Triggiani, R.}
\newblock An abstract semigroup approach to the third-order
  {Moore}–{Gibson}–{Thompson} partial differential equation arising in
  high-intensity ultrasound: structural decomposition, spectral analysis,
  exponential stability.
\newblock {\em Mathematical Methods in the Applied Sciences 35}, 15 (2012),
  1896--1929.
\newblock Publisher: Wiley Online Library.

\bibitem{wilke}
{\sc Meyer, S., and Wilke, M.}
\newblock Optimal regularity and long-time behavior of solutions for the
  {Westervelt} equations.
\newblock {\em Applied Mathematics and Optimization 64\/} (2011), 257--271.
\newblock Publisher: Springer Verlag.

\bibitem{pellicer2017}
{\sc Pellicer, M., and Said-Houari, B.}
\newblock Wellposedness and {Decay} {Rates} for the {Cauchy} {Problem} of the
  {Moore}–{Gibson}–{Thompson} {Equation} {Arising} in {High} {Intensity}
  {Ultrasound}.
\newblock {\em Applied Mathematics \$\&\$ Optimization 80}, 2 (Dec. 2017),
  447--478.
\newblock Publisher: Springer Science and Business Media LLC.

\bibitem{ap3}
{\sc Straughan, B.}
\newblock {\em Heat waves}.
\newblock Springer Science \& Business Media, 2011.

\end{thebibliography}

\end{document}